\documentclass[]{amsart}

\usepackage{a4wide}
\usepackage[utf8]{inputenc}
\usepackage[margin=1.3in]{geometry}
\usepackage[titletoc,title]{appendix}
\usepackage[dvipsnames]{xcolor}
\usepackage{amssymb,amsfonts,amsthm,amscd,stmaryrd}
\usepackage{amsmath}
\usepackage{fancyhdr}
\usepackage{sidecap}
\usepackage{booktabs}
\usepackage{graphicx}
\usepackage{float}
\usepackage{mathtools}
\usepackage{indentfirst}
\usepackage{mathrsfs}
\usepackage{listings}
\usepackage{amsbsy}
\usepackage{amsfonts}
\usepackage{latexsym}
\usepackage{amsopn}
\usepackage{amsxtra}
\usepackage{euscript}
\usepackage{ dsfont }
\usepackage{ esint }
  \usepackage{mathrsfs}
 \usepackage{latexsym}
 \usepackage{bbm}
 \usepackage{mathtools}
 \usepackage[font=small,format=hang,labelfont={sf,bf}]{caption}
 \usepackage{epsfig}
 \usepackage{subfig}
 \usepackage{url}
 \usepackage{varioref}
 \usepackage{bm}
\usepackage{tikz}

\usepackage[dvipsnames]{xcolor}
\theoremstyle{definition}
\newtheorem{defin}{Definition}[section]
\newtheorem*{defin*}{Definition}
\newtheorem{oss}[defin]{Remark}
\newtheorem*{oss*}{Remark}

\newtheorem*{es*}{Example}
\newtheorem*{dim*}{}

\theoremstyle{plain}

\newtheorem{teo}[defin]{Theorem}
\newtheorem*{teo*}{Theorem}
\newtheorem*{cor*}{Corollary}
\newtheorem{lemma}[defin]{Lemma}
\newtheorem*{lemma*}{Lemma}
\newtheorem{prop}[defin]{Proposition}
\newtheorem*{prop*}{Proposition}
\allowdisplaybreaks

\newcommand{\N}{\mathbb{N}}

\newcommand{\R}{\mathbb{R}}
\let\H\undefined
\newcommand{\H}{\mathbb{H}}
\let\P\undefined
\newcommand{\P}{P_{\phi}}
\newcommand{\h}{^{(h)}}
\newcommand{\ud}{\,\textnormal{d}}
\newcommand{\udH}{\ud \mathcal{H}^{N-1}}

\newcommand{\virg}[1]{``#1''}
\newcommand{\bd}{\partial}
\newcommand{\weakstar}{\overset{\ast}{\rightharpoonup}}
\renewcommand{\div}{\textnormal{div}}
\newcommand{\dist}{\textnormal{dist}}
\newcommand{\sd}{\textnormal{sd}}
\newcommand{\e}{\varepsilon}

\usepackage[maxbibnames=99]{biblatex} 
\begin{document}
\title[Inhomogeneous Mean Curvature Flows]{Minimizing Movements for Anisotropic and Inhomogeneous Mean Curvature Flows}

\author[Antonin Chambolle]{Antonin Chambolle}
\address[Antonin Chambolle]{CEREMADE department, Université Paris-Dauphine, Université PSL, CNRS, pl. du Maréchal de Lattre de Tassigny, 75016 Paris, France}

\author[Daniele De Gennaro]{Daniele De Gennaro} \address[Daniele De Gennaro]{CEREMADE department, Université Paris-Dauphine, Université PSL, CNRS, pl. du Maréchal de Lattre de Tassigny, 75016 Paris, France}

\author[Massimiliano Morini]{Massimiliano Morini}
\address[Massimiliano Morini]{Dipartimento di Scienze Matematiche, Fisiche e Informatiche, Università di Parma, Parco Area delle Scienze 7/A, Parma, Italy}

\date{}

\maketitle
	
\begin{abstract}
In this paper we address anisotropic and inhomogeneous mean curvature flows with forcing and mobility, and show that the minimizing movements scheme converges to level set/viscosity solutions and to distributional solutions \textit{à la}  Luckhaus-Sturzenhecker to such flows, the latter result holding in low dimension and conditionally to the convergence of the energies. By doing so we generalize recent works concerning the evolution by mean curvature by removing  the hypothesis of translation invariance, which in the classical theory allows one to simplify many arguments. 
\end{abstract}

 \section{Introduction}
In this paper we deal with the anisotropic, inhomogeneous mean curvature flow with forcing and mobility. By inhomogeneous we mean that the flow is driven by surface tensions depending on the position in addition to the orientation of the surface. The evolution  of sets $t\mapsto E_t\subseteq \R^N$ considered is (formally) governed by the law
\begin{equation}\label{smooth law}
    V(x,t)= \psi(x,\nu_{E_t}(x))\left(  -H^\phi_{E_t}(x)+f(x,t)  \right),\quad x\in\bd E_t,\ t\in (0,T),
\end{equation}
where $V(x,t)$ is the (outer) normal velocity of the boundary $\bd E_t$ at $x$, $\phi(x,p)$ is a given anisotropy representing the \textit{surface tension}, $H^\phi$ is the \textit{anisotropic mean curvature} of $\bd E_t$ associated to $\phi$, $\psi(x,p)$ is an anisotropy evaluated at the outer unit normal $\nu_{E_t}(x)$ to $\bd E_t$ which represents a \textit{velocity} modifier (also called the \textit{mobility} term), and $f$ is the \textit{forcing} term. We will be mainly concerned with smooth anisotropies (and the regularity assumptions will be made precise later on): in this case, the curvature $H^\phi$ is the first variation of the anisotropic and inhomogeneous  perimeter associated to the anisotropy $\phi$ (in short, $\phi-$perimeter) defined as
\begin{equation}
    \P(E):=\int_{\bd^* E}\phi(x,\nu_E(x))\udH(x)
\end{equation}
for any set $E$ of finite perimeter (where $\bd^* E$ denotes the reduced boundary of $E$) and, if $E$ is sufficiently smooth, it takes the form
\[ H^\phi_E(x)=\div (\nabla_p \phi(x,\nu_{E}(x))), \]
where with $\nabla_p$ we denote the gradient made with respect to the second variable. Note that  evolution~\eqref{smooth law} can be red as the motion of sets in $\R^N$, when the latter is endowed with the Finsler metric induced by the anisotropy (see Remark \ref{rmk rel geom}). Equation \eqref{smooth law} is relevant in Material Sciences, Crystal Growth, Image Segmentation, Geometry Processing and other fields see e.g.  \cite{AllCah,Des,Gur,Set,Tau}.

The mathematical literature for inhomogeneous mean curvature flows is not as extensive as in the homogeneous case, mainly due to the difficulties arising from the lack of translational invariance. Indeed, assuming that the evolution is invariant under translations allows to simplify many arguments used in the classical proofs of, for example, comparison results and estimates on the speed of evolution. In the homogeneous case the well-posedness theory is nowadays well established and quite satisfactory, both in the local and nonlocal case, and even in the much more challenging crystalline case (that is, when the anisotropy $\phi$ is piecewise affine) see \cite{AlmChaNov,AlmTayWan,BelPao,ChaMorNovPon,ChaMorPon12,ChaMorPon15,CheGigGot,GigPoz22, LauOtt,LucStu,MugSeiSpa} to cite a few.  Concerning the inhomogeneous mean curvature flow, we cite \cite{Hui86,HuiPol} where the short time  existence of smooth solutions on manifolds is shown, and \cite{GigGotIshSat,Ilm}, where the viscosity level set approach  (introduced for the homogeneous evolution in \cite{CheGigGot,EvaSpr}) is extended, respectively, to the equation \eqref{smooth law} and to the Riemannian setting.

In the present work we implement the  {minimizing} movement approach à la Almgren-Taylor-Wang  {(in short, ATW scheme)} \cite{AlmTayWan}  to prove existence via approximation of a level set solution to the generalized anisotropic and inhomogeneous motion \eqref{smooth law}. To carry on this scheme (which has only been sketched in \cite{BelPao}, but lacks a formal proof) we {gain insights from} \cite{ChaMorPon15}.  We also show that, under the additional hypothesis of convergence of the energies \eqref{hp per} and low dimension  \eqref{hp dim}(which are nowadays classical for this approach), the same approximate solutions provide in the limit a suitable notion of \virg{BV-solutions}, also termed distributional solutions, see \cite{LucStu,MugSeiSpa}.

 {There are many more concepts of weak solution for the mean curvature flow. In particular, we cite the diffuse-interface approximation provided by the Allen-Cahn equation \cite{EvaSonSou,Ilm93,HenLau,LauStiUll} and the threshold dynamic scheme \cite{MerBenOsh,EseOtt} (see also the  relative entropy methods of \cite{LauOtt}).  Other recent results concern the weak-strong uniqueness problem, which consists in proving that weak solutions coincide with the smooth ones as long as the latter exist. After classical works concerning viscosity solutions, a new definition of \virg{BV-solution} (whose existence is proved via the Allen-Cahn approximation scheme) allows the authors in  \cite{HenLau,LauStiUll} to prove  weak-strong uniqueness for isotropic and anisotropic mean curvature flows. This result  is based upon the so-called \textit{optimal dissipation inequality} satisfied by their weak solution. In general, it is very difficult to say if the ATW scheme could satisfy such a property, mainly because of the \virg{degeneracy} of the dissipation term in the incremental problem defined via the distance function. Even if all these results concern the translationally invariant case, a study of some of these properties in the inhomogeneous setting seems very interesting and challenging. }

 {Other remarks on possible research directions} are the following. To begin with, the new arguments which are used to compensate the lack of translation invariance are based on the locality of the anisotropic curvature $H^\phi$ associated with a smooth anisotropy $\phi$. This implies that the proofs are not straightforwardly adaptable to the so-called \virg{variational curvatures} considered in \cite{ChaMorPon15}, which are non-local in nature.  On the other hand, since the  crystalline curvatures are highly nonlocal and degenerate operators (see e.g. \cite{ChaMorNovPon, ChaNov22}), they do not fall in the theory constructed in the present work. In principle,  it would be possible to follow the same perturbative study conducted in \cite{ChaMorNovPon} in order to prove at least existence for an inhomogeneous and crystalline mean curvature flow. However, a satisfactory characterization of the limiting motion equation bearing a comparison principle is lacking so far.

This work can be seen as a first step towards  constructing a general theory of motions driven by   {non-translationally invariant} and possibly nonlocal curvatures, in the spirit of \cite{ChaMorPon15}.

\subsection{Main results}
Now briefly recall the minimizing movements procedure in order to state the main results of the paper. Given an initial bounded set $E_0$ and a parameter $h>0$, we define the\textit{ discrete flow} $E_t\h:=T_{h,t-h}E_{t-h}\h $ for any $t\ge h$ and $E_t\h=E_0$ for $t\in [0,h),$ where the functional $T_{h,t}$ is defined for $t\ge 0$ as follows: for any bounded set $E$ we set $T_{h,t} E$ (or, sometimes, $T_{h,t}^- E$) as the minimal solution to the problem 
\begin{equation}
    \min\left\lbrace  P_\phi(F)+\int_F \left(\dfrac{\sd^\psi_E(x)}h+\fint_{[\frac th]h}^{[\frac th]h+h} f(x,s)\ud s \right)\udH(x)\ :\ F \text{ is measurable}\right\rbrace,
\end{equation}
where $\sd^\psi_E(x)$ is the signed geodesic distance between $x$ and $E$ induced by the anisotropy $\psi$ (see \eqref{def sd} for the precise definition) and $[s]=\max \{ n\le s, \ n\in\N\cup\{0\} \}$ denotes the integer part of a non-negative real number $s\in[0,+\infty)$. We will then define  {$T^+_{h,t} E$} as the maximal solution to the problem above. Any $L^1-$limit point as $h\to 0$ of the family $\{E_t\h\}_{t\ge 0}$ will be called a\textit{ flat flow}. In the whole paper we will assume that
\begin{equation}\tag{H0}
    \begin{split}\label{standing hp}
        & \phi\in\mathscr E \text{ (see Definition \ref{def reg ell integr}) and }  \psi \text{ is an anisotropy as in Definition \ref{def anisotropy},  }\\
        & \forall t\in [0,+\infty)  \text{ it holds } f(\cdot,t)\in C^0(\R^N), \|f\|_{L^\infty(\R^N\times[0,+\infty))}<\infty. 
    \end{split}
\end{equation}
With more effort one could weaken the hypothesis and require $\int_0^tf(\cdot,s)\ud s$ to be continuous (see \cite{ChaNov08}). For the sake of simplicity we will require the global-in-time boundedness. We prove existence and H {\"o}lder regularity for flat flows.
\begin{teo}[Existence of flat flows]
    \label{theorem existence flat flow}
    Let $E_0$ be a bounded set of finite perimeter and $\phi,\psi,f$ satisfy \eqref{standing hp}.  {Fix $T>0$}. For any $h>0$, let $\{E_t^{(h)}\}_{ {t\in [0,T)}}$ be a discrete flow with initial datum $E_0$. Then, there exists a family of sets of finite perimeter  {$\{E_t\}_{t\in [0,T)}$} and a subsequence $h_k\searrow 0$ such that
    \[ E_t^{(h)}\to  E_t\quad \text{in } L^1, \]
    for a.e.  {$t\in [0,T)$}. Such flow satisfies the following regularity property:  {there exists a constant $c$, depending on $T$, such that} for every  $0\le s\le t<T$,
    \begin{align*}
        |E_s\triangle E_t|&\le c|t-s|^{1/2},\\
        P_\phi(E_t)&\le P_\phi(E_0)+c.
    \end{align*}
\end{teo}

Subsequently, we will show that flat flow {s} are distributional solutions, as defined  in \cite{LucStu}. We will require additional hypothesis: firstly, low dimension \eqref{hp dim} (linked to the complete regularity of the $\phi-$perimeter minimizer, compare \cite{LucStu,MugSeiSpa}), moreover
 \begin{align}
       &\exists \,c_\psi>0\ \text{s.t. } |\psi (x,v)-\psi(y,v)|\le c_\psi |x-y|,\quad \forall x,y\in\R^N, v\in  {S^{N-1}} \tag{H1}\label{hp lip psi 1},\\
    &f\in C^0(\R^N\times [0,\infty)])\tag{H2}\label{hp f 1}.
 \end{align}

\begin{teo}[Existence of distributional solutions] \label{theorem existence distributional solutions}
     Assume \eqref{standing hp}, \eqref{hp lip psi 1}, \eqref{hp f 1} and  \eqref{hp dim}. For any $T>0$, if
    \begin{equation}\label{hp per}
        \lim_{k\to \infty}\int_0^T P_\phi(E_t^{(h_k)})=\int_0^T P_\phi(E_t),
    \end{equation}
    then $\{E_t\}_{ {t\in [0,T]}}$ is a distributional solution \eqref{smooth law} with initial datum $E_0$ in the following sense: 
    \begin{itemize}
        \item[(1)] for a.e. $t\in [0,T)$ he set $E_t$ has weak $\phi-$curvature $H^\phi_{E_t}$ (see \eqref{def curvature variational} for details) satisfying 
        \[ \int_0^T\int_{\bd ^*E_t}   |H^\phi_{E_t}|^2<\infty; \]
        \item[(2)] there exist $v:\R^N\times (0,T)\to \R$ with $ \int_0^T\int_{\bd ^*E_t}   v^2\udH\ud t<\infty$ and  $v(\cdot,t)\big\lvert_{\bd E_t}\in L^2(\bd E_t)$ for a.e. $t\in [0,T)$, such that 
        \begin{align}
            -\int_0^T\int_{\bd^*E_t} v\eta\udH\ud t &= \int_0^T\int_{\bd^* E_t}  \left(H^\phi_{E_t} - f \right) \eta\udH\ud t   \label{legge curvatura}\\
            \int_0^T\int_{E_t} \bd_t \eta \ud x\ud t+ \int_{E_0} \eta(\cdot,0)\ud x&= -\int_0^T\int_{\bd^*E_t} \psi(\cdot,\nu_{E_t})  v\eta \udH \ud t, \label{legge velocita}
        \end{align}
        for every $\eta\in C^1_c(\R^N \times  {[0,T)}).$
    \end{itemize}
\end{teo}
The definitions $1),2)$ extend to our case the definition of $BV$-solutions of \cite{LucStu} and the distributional solutions of \cite{MugSeiSpa}. We recall that hypothesis \eqref{hp per} ensures that the evolving sets avoid the so-called \virg{fattening} phenomenon. It is known that this hypothesis is satisfied in the case of evolution of convex or mean-convex sets, see e.g. \cite{ChaNov22,DePLau,FucLau}, but in general is not known under which general hypothesis it is valid. We also remark that the proof of the theorem above provides a  {detailed} proof of \cite[Theorem 3.2]{ChaNov22}, which had only been sketched.  {Moreover, we bypass the use of a  Bernstein-type result (which is usually employed) by a double blow-up technique.}

In the second part of the work we will focus on the level set approach. Briefly, given an initial compact set $E_0$, we set $u_0$ such that  $\{u_0 \ge 0\}=E_0$ and we look for a solution $u$ in the viscosity sense (in a sense made precise in Definition \ref{def visco sol}) to 
\begin{equation}\label{intro cauchy}
	\begin{cases}
		\bd_t u+\psi(x,-\nabla u)\left( \div \nabla_p \phi(x,  \nabla u(x))-f(x,t) \right)=0\\
		u(\cdot,t)=u_0.
	\end{cases}
\end{equation}
Classical remarks ensure that any level set $\{u \ge s\}$ is evolving following the mean curvature flow \eqref{smooth law}. To prove existence for \eqref{intro cauchy} we use an approximating procedure.  For $h>0$ and $t\in (0,+\infty)$ we set iteratively  $u_h^\pm(\cdot,t)=u_0$ for $t\in [0,h)$ and for $t\ge h$
\begin{equation*}
	\begin{split}
		u^+_h(x,t)&:=\sup\left\lbrace s\in\R\ :\ x\in T_{h,t-h}^+\{ u_h^+(\cdot,t-h)\ge s \}\right\rbrace   \\
		u_h^-(x,t)&:=\sup\left\lbrace s\in\R\ :\ x\in T_{h,t-h}^-\{ u_h^-(\cdot,t-h)> s \}\right\rbrace,
	\end{split}
\end{equation*}
where the operator $T_{h,t}^\pm$ has been previously introduced. We remark that these are  maps piecewise constant in time, since $T^\pm_{h,t}=T^\pm_{h,[t/h]h}$, which are  only upper and lower semicontinuous in space respectively. Then, we will pass to the limit $h\to 0$ on the families $\{u_h^\pm\}_h$ to find  functions $u^+, u^-$ which are viscosity sub {-} and supersolution respectively of equation \eqref{intro cauchy}.  Passing to the limit as $h\to 0$ in our case is not straightforward. The main issue is that we do not have an uniform estimate on the modulus of continuity of the functions $u_h$ (compare \cite{ChaMorPon15}) and thus we can not pass to the (locally) uniform limit of the sequence. (More precisely, our best estimate contained in Lemma \ref{lemma estimate decay balls} decays too fast as $h\to 0$ to provide any useful information). Nonetheless, motivated by \cite{BarSonSou,Bar-rev,BarSou} we can define the half-relaxed limits 
\begin{equation}\label{def sottosoluzione}
	\begin{split}
		u^+(x,t):=& \sup_{(x_h,t_h)\to (x,t)}\limsup_{h\to 0} u_h^+(x_h,t_h)   \\
		u^-(x,t):=& \inf_{(x_h,t_h)\to (x,t)}\liminf_{h\to 0} u_h^-(x_h,t_h),
	\end{split}
\end{equation}
and prove that the functions defined above are sub {-} and supersolutions, respectively, to \eqref{intro cauchy}. The main difficulty in this regard is that we need to work with  just semicontinuous functions in space, as in the translationally invariant setting one can easily  prove the uniform equicontinuity of the approximating sequence. We prove the following.
\begin{teo}\label{teo sol viscosa}
    Assume \eqref{standing hp}, \eqref{hp lip psi 1} and $f\in C^0(\R^n\times [0,+\infty))$. The function $u^+$ (respectively $u^-$) defined in \eqref{def sottosoluzione} is a viscosity subsolution (respectively a viscosity supersolution) of \eqref{intro cauchy}.
\end{teo}
Thanks to the results of \cite{CheGigGot} we then prove that, under the additional hypothesis
\begin{equation}\tag{H3}\label{hp 2}
    \begin{split}
        &\nabla_{x}\nabla_p \phi(\cdot,p) \text{ and } \nabla^2_p\phi(\cdot,p) \text{ are Lipschitz, uniformly for } p\in  {S^{N-1}}\\ 
        &\nabla_p^2 \phi^2(x,p) \text{ is  uniformly elliptic in } p, \text{uniformly in }x \\
	& \psi(\cdot,p) \text{ Lipschitz continuous, uniformly in } p\\
	& f(\cdot,t)\text{ Lipschitz continuous, uniformly in } t,
    \end{split}
\end{equation}
the following uniqueness result holds.
\begin{teo}\label{existence and regularity of viscosity solutions}
	Assume \eqref{standing hp} and \eqref{hp 2}. If $u_0$ is a continuous function which is spatially constant outside a compact set, equation \eqref{intro cauchy} with initial condition $u_0$ admits a unique continuous viscosity solution $u$ given by \eqref{def sottosoluzione}.  In particular, $u^+=u^-=u$ is the unique continuous viscosity solution to \eqref{intro cauchy} and $u_h^\pm\to u$ as $h\to 0$, locally uniformly.
\end{teo}
The previous result  yields  a proof of consistency between the level set approach and the minimizing movements one to study  the evolution \eqref{smooth law}.  We recall that it has been  established for the classical mean curvature flow in \cite{Cha}, in the anisotropic but homogeneous case in \cite{EtoGigIsh} and in a very general nonlocal setting in \cite{ChaMorPon15}.

\section{Preliminaries}
We start introducing some notations. We consider $0\in\N$. We will use both $B_r(x)$ and $B(x,r)$ to denote the Euclidean ball in $\R^N$ centered in $x$ and of radius $r$;  with $B^{N-1}_r(x)$ we denote the Euclidean ball in $\R^{N-1}$ centered in $x$ and of radius $r$; with $ {S^{N-1}}$ we denote the sphere $\bd B_1(0)\subseteq \R^N$;  {with $Sym_N$ the symmetric real matrices of size $N\times N$}.
In the following, we will always speak about measurable sets and refer to a set as the union of all the points of density $1$ of that set i.e. $E=E^{(1)}.$ If not otherwise stated, we implicitly assume that the function spaces considered are defined on $\R^N$,  e.g $L^\infty=L^\infty(\R^N)$;  {the space $C^0$ denotes the space of continuous functions.} Moreover, we often drop the measure with respect to which we are integrating, if clear from the context.

\begin{defin}\label{def anisotropy}
   We define \textit{anisotropy} (sometimes defined as an \textit{elliptic integrand}) a function $\psi$ with the following properties:  $\psi(x,p):\R^N\times \R^{N}\to [0,+\infty)$ is a continuous function, which is convex and positively 1-homogeneous in the second variable,  such that
\begin{equation*}
    \dfrac 1{c_\psi}|p|\le \psi(x,p)\le c_\psi |p|
\end{equation*}
for any point $x\in\R^N$ and vector $p\in\R^N$.
\end{defin}
We remark that, as standard, we define a real function $f$ positively 1-homogeneous if for any $\lambda\ge 0,$ it holds $f(\lambda x)=\lambda f (x)$. In particular, the anisotropies that we will consider are not symmetric. In the following, we will always denote the gradient of an anisotropy with respect to the first (respectively second) variable as $\nabla_x\psi$ (respectively $\nabla_p\psi$).  We then recall the definition of some well-known quantities (see \cite{BelPao}). Define the polar function of an anisotropy $\psi$, denoted with $\psi^\circ$, as 
\begin{equation}\label{CS type ineq}
    \psi^\circ(\cdot,\xi):=\sup_{p\in \R^N}\left\lbrace \xi\cdot p\ :\ \psi(\cdot,p)\le 1\right\rbrace.   
\end{equation}
Using the definition it is easy to see that for all $ p,\xi\in\R^N$ it holds
\[ \psi(\cdot,p)\psi^\circ(\cdot,\xi)\ge p\cdot \xi,\quad  -\psi(\cdot,-p)\psi^\circ(\cdot,\xi)\le  p\cdot \xi. \]
Furthermore, one can prove that (see \cite{BelPao})  {for $p\neq 0$}
\[ \psi^\circ(\nabla_p\psi)=1,\ \psi(\nabla_p\psi^\circ)=1,\  (\psi^\circ)^\circ=\psi.  \]
We define for any $x,y\in\R^N$ the geodesic distance induced by $\psi,$ or $\psi-$distance in short, as
\[ \dist^\psi(x,y):=\inf \left\lbrace  \int_0^1 \psi^\circ(\gamma(t),\dot\gamma(t))\ud t \ :\ \gamma\in W^{1,1}([0,1];\R^N), \gamma(0)=x, \gamma(1)=y\right\rbrace. \]
We remark that this function is not symmetric in general. We  define the \textit{signed distance function} from a closed set $E\subseteq \R^N$ as 
\begin{equation}\label{def sd}
    \sd^\psi_E(x):=\inf_{y\in E}\dist^\psi(y,x)-\inf_{y\notin E}\dist^\psi(x,y),
\end{equation}
so that $\sd^\psi_E\ge 0$ on $E^c$ and $\sd^\psi_E\le 0$ in $E$. We remark that the bounds stated in Definition \ref{def anisotropy}  {imply}
\begin{equation}\label{inclusione psi palle}
    \dfrac{1}{c_\psi}\dist\le \dist^\psi\le c_\psi \dist,
\end{equation}
where here and in the following we will denote with $\dist,\sd$ the Euclidean distance and signed distance function respectively.  We define the $\psi-$balls as the balls associated to the $\psi-$distance, that is
\[ B^\psi_\rho(x):=\{ y\in\R^N : \text{dist}^\psi(y,x)<\rho  \}, \]
which in general are not convex nor symmetric. 

\begin{defin}\label{def reg ell integr}
 {We say that an anisotropy $\phi$ is a \textit{regular elliptic integrand}, and write $\phi\in \mathscr E$, if there exists two constants $\lambda\ge 1,\,l\ge 0$ such that} if $\phi(x,\cdot)\big\lvert_{ {S^{N-1}}}\in C^{2,1}(  {S^{N-1}})$ and for every $x,y,e\in\R^N,\, \nu,\nu'\in    {S^{N-1}}$ one has:
\begin{align*}
    \dfrac 1\lambda\le \phi(x&,\nu)\le \lambda, \\
    |\phi(x,\nu)-\phi(y,\nu)|+|\nabla_p& \phi(x,\nu)-\nabla_p \phi(y,\nu)|\le l|x-y|\\
    |\nabla_p\phi(x,\nu)|+\|\nabla_p^2 \phi(x,\nu)\|+ &\dfrac{\|\nabla_p^2 \phi(x,\nu)-\nabla_p^2 \phi(x,\nu')\|}{|\nu-\nu'|}\le \lambda\\
    e\cdot \nabla_p^2\phi(x,\nu)[e]&\ge \dfrac{|e-(e\cdot\nu)\nu|^2}{\lambda}.
\end{align*}
\end{defin} Given any set of finite perimeter $E$, one can define the  $\phi-$perimeter $\P$ as follows
\[ \P(E):=\int_{\bd^* E}\phi(x,\nu_E(x))\udH(x), \]
where $\bd^* E$ is the reduced boundary of $E$ and $\nu_E$ is the measure-theoretic outer normal, see \cite{Mag-book} for further references on sets of finite perimeter. The  {$\phi-$}perimeter of a set of finite perimeter $E$ in an open set $A$ is defined  as 
\[ \P(E;A):=\int_{\bd^* E\cap A}\phi(x,\nu_E(x))\udH(x). \]
We remark that, by definition of regular elliptic integrand, for any  set $E$  {of finite perimeter} it holds $$\frac 1\lambda P(E)\le P_\phi(E)\le \lambda P(E).$$ Some additional remarks on this definition can be found in \cite{DePMag1}. We just recall the submodularity property of the $\phi-$perimeter,  {which can be  proved for instance by using the formulae for the reduced boundary and measure-theoretic normal of union and intersection of sets of finite perimeter (see \cite{Mag-book})}.
\begin{prop}[Submodularity property]
    For any two sets  $E,F\subseteq\R^N$  of finite perimeter, one has
    \begin{equation}\label{submod}
         \P(E\cup F)+\P(E\cap F)\le \P(E)+\P(F).
    \end{equation}
\end{prop}
Moreover, by homogeneity, \eqref{CS type ineq} and recalling that for any set  $E$ of finite perimeter it holds $D\chi_E=-\nu_E\udH{\big\lvert_{\bd^* E}}$ we have the following equivalent definitions
\begin{align}
    \P(E)&=\sup \left\lbrace \int_{\R^N}-D\chi_E\cdot \xi \ :\ \xi\in C^1_c(\R^N;\R^N), \phi^\circ(\cdot,\xi)\le 1 \right\rbrace\label{P via calibration}\\
    &=\sup \left\lbrace \int_{E}\div\, \xi\udH \ :\ \xi\in C^1_c(\R^N;\R^N), \phi^\circ(\cdot,\xi)\le 1 \right\rbrace\nonumber.
\end{align}
Concerning the regularity property of the $\phi-$perimeter minimizers, we refer to \cite{SchSimAlm}. We just recall the following results. Given two anisotropies $\phi,\psi\in\mathscr E$, we define the \virg{distance} between them as
\begin{align*}
    &\dist_{\mathscr E}(\phi,\psi):=\sup \{ |\phi(x,p)-\psi(x,p)|\\
    &+ |\nabla_p \phi(x,p)-\psi(x,p)|  +  |\nabla_p^2 \phi(x,p)-\nabla_p^2\psi(x,p)| : x\in\R^N, p\in  {S^{N-1}}\}, 
\end{align*}
where $|\cdot|$ denotes the Euclidian norm. Given $\phi\in\mathscr E$, we recall that $E$ is a $0-$minimizer for the $\phi-$perimeter if for any $x\in\R^N,r>0$
\[  P_\phi(E;B_r(x))\le P_\phi(F;B_r(x))  \]
for every $F\subset \R^N$ such that $F\triangle E\subset\joinrel\subset B_r$.  Then, some regularity properties of minimizers of $\phi-$perimeter can be found in  the theorems of part $II.7$ and $II.8$ in \cite{SchSimAlm}, which are recalled below.

\begin{teo}
    Assume $\phi\in \mathscr E$. Then,  for any 0-minimizer $E$ of the $\phi-$perimeter,  the reduced boundary $\bd^* E$ of the set $E$ is of class $C^{1,1/2}$ and the singular set $\Sigma:=\bd E\setminus \bd^* E$ satisfies
    \[  \mathcal H^{N-3}(\Sigma)=0. \]
\end{teo}

\begin{teo}\label{teo reg ellipt integr}
   Let $m>0,\alpha\in (0,1)$. Then, there exists $\e=\e(m,\alpha)>0$ with the following property:  let $\phi=\phi(p)\in\mathscr E$, $\phi\in C^{3,\alpha}(\R^N\setminus \{0\})$ with 
   \[ \| \phi|_{S^{N-1}} \|_{C^{3,\alpha}}\le m\ \text{and } \dist_{\mathscr E}(\phi,|\cdot|)\le \e.  \]
    Then, for any 0-minimizer $E$ of the $\phi-$perimeter,  the reduced boundary $\bd^* E$ of the set $E$ is of class $C^{1,1/2}$ and the singular set $\Sigma:=\bd E\setminus \bd^* E$ satisfies
    \[  \mathcal H^{N-7}(\Sigma)=0. \]
\end{teo}

We sum up these hypotheses that yield the complete regularity of minimizers of parametric elliptic integrands: 
\begin{equation}
    \begin{split}\label{hp dim}
        &\text{either }  \phi\in \mathscr E \text{ and }N\le 3,\\
        &\text{or } N\le 7  \text{ and the hypotheses of Theorem \ref{teo reg ellipt integr} are satisfied}.  
    \end{split}
\end{equation}

\subsection{The first variation of the $\phi-$perimeter}

In this section we compute the first variation of the $\phi$-perimeter and define some additional operators associated to it.

Assume $E$ is of class $C^2$. Let $X$ be a  {smooth and compactly supported} vector field and assume $\Psi(x,t)=:\Psi_t(x)$ is the associated	 flow. To simplify the notation, we write
\[ \nu(x,t)=\nabla_x \sd_{\Psi(E,t)}(x). \]
By classical formulae (see e.g. \cite{CagMorMor}) we can compute the following. For the sake of brevity, we avoid writing the evaluation $\phi=\phi(x,\nu_E(x)) 
$, if not otherwise specified, and assume that all the integrals are made with respect to the Hausdorff $(N-1)$-dimensional measure $\mathcal H^{N-1}.$
\begin{align}
    &\dfrac{\ud}{\ud t}\Big\lvert_{t=0} P_\phi(E_t)
    = \dfrac \ud{\ud t}\Big\lvert_{t=0}    \int_{\bd E} \phi(\Psi_t(x),\nu(\Psi_t(x),t)) J\Psi_t\nonumber\\
    &=\int_{\bd E} \nabla_x\phi \cdot X + \nabla_p \phi \cdot \left( -\nabla_{\tau} (X\cdot \nu)+D\nu[X]  \right)+\phi\, \div_\tau X  \label{eq align}\\
    &= \int_{\bd E} \nabla_x\phi \cdot X + \nabla_p \phi \cdot \left( -\nabla_{\tau} (X\cdot \nu)+D\nu[X]  \right) + \div_\tau (\phi X) - \nabla\phi\cdot X + (\nabla\phi\cdot \nu)(X\cdot \nu)\nonumber\\
    &=\int_{\bd E} \nabla_x\phi \cdot X + \nabla_p \phi \cdot \left( -\nabla_{\tau} (X\cdot \nu)+D\nu[X]  \right) -	\nabla_x\phi\cdot X -D\nu[\nabla_p \phi] \cdot X \nonumber\\
    &\quad +\div_\tau (\phi X) + (\nabla\phi\cdot \nu)(X\cdot \nu)\nonumber\\
    &=\int_{\bd E} -\nabla_p \phi \cdot \nabla_\tau (X\cdot \nu) + (\nabla_x\phi\cdot \nu)(X\cdot \nu) + \left( D\nu[	\nabla_p \phi]\cdot \nu \right) (X\cdot \nu) +\div_\tau (\phi X)\nonumber\\
    &=\int_{\bd E} \div_\tau \left(\nabla_p \phi(X\cdot \nu)\right)-\nabla_p \phi \cdot \nabla_\tau (X\cdot \nu)  + (X\cdot \nu)(\nabla_x \phi\cdot \nu)\nonumber\\
    &=\int_{\bd E} (\div_\tau \nabla_p \phi)(X\cdot \nu)+\nabla_p \phi\cdot \nabla_\tau (X\cdot \nu) - \nabla_p \phi\cdot \nabla_\tau (X\cdot \nu) + (\nabla_x\phi\cdot \nu)(X\cdot \nu)\nonumber\\
    &=\int_{\bd E} (X\cdot \nu) \left( \div_\tau \nabla_p \phi + \nabla_x \phi\cdot 	\nu \right)=\int_{\bd E}(X\cdot \nu ) \,\div \nabla_p \phi\nonumber
\end{align}
where the last equality follows from the definition of $\div_\tau$ and the fact that $\phi$ is $1-$homogeneous with respect to the $p$ variable, since
\begin{align*}
    \div \nabla_p \phi&=\div_\tau \nabla_p \phi +\sum_i \nu_i  \left( {\bd_{x_i}}\nabla_p \phi\right)[\nu]\\
    &=\div_\tau \nabla_p \phi +\sum_i \nu_i \nabla_p (\bd_{x_i}\phi)\cdot \nu +\nu\cdot \left(  \nabla^2_p \phi D\nu  \right)[\nu]\\
    &=\div_\tau \nabla_p \phi +\nabla_x \phi\cdot \nu.
\end{align*}
Therefore, we define the first variation of a $C^2-$regular set $E$, induced by the vector field $X$, as 
\begin{equation}\label{first variation perimeter}
    \delta P_\phi (E)[X\cdot \nu]:=
    \int_{\bd E}(X(x)\cdot \nu(x) ) \,\div \nabla_p \phi(x,\nu(x))\udH(x)
\end{equation} 
and the $\phi-$curvature of the set $E$ as
\begin{equation}
    \label{def curvature}
    H^\phi_E(x):=
    \div \nabla_p \phi(x,\nu(x)).
\end{equation}
If we now consider equation \eqref{eq align}, we develop the tangential gradient to find 
\begin{align*}
    \nabla_p\phi \cdot (-\nabla_\tau (X\cdot \nu)+D\nu[X])=\nabla_p\phi \cdot (-\nabla_\tau X[\nu]-D\nu[X] +D\nu[X])=0.
\end{align*}
This shows that for any set $E$ of class $C^2$  it holds 
\begin{equation*}
    \delta P_\phi (E)[X\cdot \nu]:=\int_{\bd E}\left(\nabla_x \phi \cdot X+ \phi\, \div_\tau X\right) \udH,
\end{equation*} 
where we dropped the evaluation of $\phi$ at $(x,\nu_E(x))$. We remark that the expression on the right hand side makes sense even if the set $E$ is just of finite perimeter. Defining the $\phi-$divergence operator $\div_{\phi} $ as 
\begin{align}
    \div_{\phi} X:=\nabla_x \phi\cdot X +\phi\,\div_\tau X,
\end{align}
we are led to define the distributional $\phi-$curvature of a set $E$ of finite perimeter  as an operator $H^\phi_E\in L^1(\bd E)$ (if it exists) such that the following representation formula holds
\begin{equation}
    \label{def curvature variational}
    \int_{\bd E} \div_{\phi} X \udH = \int_{\bd E} H^\phi_E \,\nu_E\cdot X\udH,\qquad \forall X\in C^\infty_c(\R^N;\R^N).
\end{equation}
The previous computations allow to say that the distributional $\phi-$curvature can be expressed as \eqref{def curvature} if the set is of class $C^2$. Finally, since $\phi$ is a regular elliptic integrand, one can prove the following monotonicity result.

\begin{lemma}
    Let $E,F$ be two $C^2$ sets of finite $\phi-$perimeter  with $E\subseteq F$, and assume that $x\in \bd F\cap \bd E$: then $H^\phi_F(x)\le H^\phi_E(x).$
\end{lemma}

\begin{proof}
    Since the anisotropy is smooth, we can expand the curvature formula \eqref{def curvature} as 
    \begin{equation}
        \label{curvature expanded}
        H^\phi=\text{tr}\left( \nabla_x\nabla_p \phi(x,\nu) +\nabla^2_p \phi(x,\nu) D\nu \right)
    \end{equation}
    and compare $H^\phi_E$ with $H^\phi_F$. We consider separately the two terms appearing in \eqref{curvature expanded}. The first one depends on $\nu$ just by the value it has at the point $x$. Therefore, since $\nu_E(x)=\nu_F(x)$ we have the equality. The second one falls in the classical framework of smooth anisotropies that do not depend on the space variable.  {Since $D\nu_F\le D \nu_E$ (as matrices) one concludes the proof.}
\end{proof}

\section{The minimizing movements approach}
In this section we follow the work of \cite{MugSeiSpa} (see also \cite{AlmTayWan,LucStu}) to prove the existence for the  mean curvature flow via the \textit{minimizing movements} approach. We recall that in the whole paper we will assume the hypothesis \eqref{standing hp}.

\subsection{The discrete scheme}
In this subsection we will define the discrete scheme approximating the weak solution of the mean curvature flow, and we shall study some of its properties.

We define the following iterative scheme. Given $h>0, f\in L^\infty(\R^N\times [0,\infty))$ and $t\ge h$, and given a bounded set of finite perimeter $F$, we minimize the energy functional 
\begin{equation}\label{problema discreto}
    \mathscr F_{h,t}^F(E)=P_\phi(E) + \dfrac 1h \int_{E} \sd^\psi_F(x)\ud x - \int_{E} F_h(x,t)\ud x
\end{equation}
in the class of all measurable sets $E\subseteq \R^N$, and where we have set 
\[ F_h(x,t):=\fint_{t}^{t+h} f(x,s)\ud s. \]
Equivalently, we could define the energy functional as 
\[ \mathscr F_{h,t}^F(E)=P_\phi(E) + \dfrac 1h \int_{E\triangle F} |\sd^\psi_{F}| - \int_{E} F_h(x,t)\ud x , \]
which agrees with \eqref{problema discreto} up to a constant. Then, we denote 
\[ T_{h,t} F=E\in\text{argmin}\ \mathscr F_{h,t}^F.\]
 {We will refer to this minimizing procedure as the \textit{incremental problem}.} It is well-known (compare \eqref{first variation perimeter} and \cite[Proposition 17.8]{Mag-book}) that a minimimum of \eqref{problema discreto} of class $C^2$ satisfies the Euler-Lagrange equation
\begin{equation}
    \label{EL equation}
    \int_{\bd  {E}} H^\phi_{ {E}} X\cdot \nu_{ {E}}\udH = -\int_{\bd  {E}}\left( \frac 1h \sd^\psi_{ {F}}(x)-F_h(x,t)  \right)X(x)\cdot \nu_{ {E}}(x)\udH(x)
\end{equation}
for all $X\in C^\infty_c(\R^N;\R^N)$. We can then define the \textit{discrete flow}, which can be seen as a discrete-in-time approximation of the mean curvature flow starting from the initial set $E_0$.  {We define iteratively the \textit{discrete flow} by setting $E_t\h=E_0$ for $t\in [0,h)$ and
\begin{equation}\label{def discrete flow}
    E_t^{(h)}=T_{h,t-h} E_{t-h}^{(h)}  =  T_{h,([\frac th]-1)h} E_{t-h}^{(h)} ,\qquad t\in[h,+\infty),
\end{equation}
where $[\cdot]$ denotes the integer part of a real number.} This section is devoted to recall and prove some estimates on the discrete flow. The first one is a well-known existence result.
\begin{lemma*}
    For any  measurable function $g:\R^N\to\R$ such that $\min\{g,0\}\in L^1_{loc},$ the problem 
    \[ \min \left\lbrace P(E)+\int_E g\ :\ E \text{ is of finite perimeter} \right\rbrace \]
    admits a solution.
\end{lemma*}
Consider now $F$ as a bounded set of finite perimeter. Then, the function $g=\sd^\psi_F/h-F_h$ is coercive, thus $\min\{ g,0 \}\in L^1$. Therefore, by the previous result and by classical arguments see \cite[Proposition 6.1]{ChaMorPon15} for a proof, one can prove the following result.
\begin{lemma}\label{existence discrete pb}
    For any given set $F$ of finite perimeter, the problem \eqref{problema discreto} admits a solution $E$, which satisfies the discrete dissipation inequality 
    \begin{equation*}
        P_\phi(E)+\dfrac 1h \int_{E\triangle F}|\sd^\psi_F| \le P_\phi(F) + \int_{E\setminus F}F_h(x,t)\ud x  - \int_{F\setminus E}F_h(x,t)\ud x .
    \end{equation*}
    Moreover, the problem \eqref{problema discreto} admits a minimal and a maximal solution.
\end{lemma}

We define $T_{h,t}^+ F $ (respectively $T_{h,t}^- F$) as the maximal (respectively minimal) solution to \eqref{problema discreto} having as initial datum $F$. In the following, whenever no confusion is possible, we shall write $T_{h,t}$ instead of $T_{h,t}^-$.

A  comparison result holds. We will consider just bounded sets as datum for the problem \eqref{problema discreto}, but the same result holds in general for unbounded sets (see also Section \ref{sect evolution unbounded sets} for the case of unbounded sets with bounded boundary). The proof of this result is classical (see e.g. \cite{ChaMorPon15}) and it is based on the submodularity of the perimeter \eqref{submod}. We will omit it.
\begin{lemma}[Weak comparison principle]\label{comparison principle}
    Assume that $F_1,F_2$ are bounded sets with $F_1 {\subset\joinrel\subset} F_2$ and consider $g_1,g_2\in L^\infty$ with $g_1\ge g_2$. Then, for any two solutions $E_i$, $i=1,2$ of the problems
    \begin{equation*}
        \min\left\lbrace \P(E)+\int_E \dfrac{\sd^\psi_{F_i}}h+g_i :  E \textnormal{ is of finite perimeter}\right\rbrace,
    \end{equation*}
    we have $E_1\subseteq E_2$. If, instead, $F_1\subseteq F_2$, then we have that the minimal (respectively maximal) solution to \eqref{problema discreto} for $i=1$ is contained in the minimal (respectively maximal)  solution to \eqref{problema discreto} for $i=2$.
\end{lemma}

We now prove the volume-density estimates for minimizers of problem \eqref{problema discreto}. This result is based on the minimality properties of almost-minimizers for perimeters induced by regular elliptic integrands (see \cite[Remark 1.9]{DePMag1} for further results). These estimates have the disadvantage that the smallness condition on the radius depends on the parameter $h$. Subsequently, we will recall a finer result in the spirit of \cite{LucStu}, where we can drop this dependence by making some restrictions on the balls considered.
\begin{lemma}\label{lemma density estimate}
    Let $g\in L^\infty$ and assume $E$ minimizes the functional 
    \[\mathscr F(F)=P_\phi(F)+\int_F g\]
    among all measurable subsets of $\R^N$. Then the density estimate
    \begin{align}
        \sigma \rho^N &\le |B_\rho (x)\cap E|\le (1-\sigma) \rho^N \nonumber \\
        \sigma \rho^{N-1}&\le P_\phi(E;B_\rho(x))\le (1-\sigma)\rho^{N-1} \label{perimeter density estimates, h dip}
    \end{align}
    holds for all $x\in \bd^* E,$ $0<\rho<( 2\lambda\|g\|_\infty)^{-1}  :=\rho_0,$ for a suitable $\sigma=\sigma(N,c_\psi,\lambda)$.
\end{lemma}
\begin{proof}
    By minimality, 
    \[ P_\phi(E)\le P_\phi(F)+\|g\|_\infty |E\triangle F| \qquad \forall F\subseteq \R^N, \]
    thus \cite[Lemma 2.8]{DePMag1} implies the thesis.	
\end{proof}

\begin{oss}\label{open close min max sols}
    We remark that the previous result allows us to choose the minimal solution to \eqref{problema discreto} to be an open set, and the maximal one to be a closed set. This follows from the fact that the density estimates imply that the boundary of any minimizer has zero measure.
\end{oss}

We now recall \cite[Lemma 3.7]{ChaMorNovPon}, which is an anisotropic version of \cite[Remark 1.4]{LucStu}. It provides volume-density estimates for minimizers of \eqref{problema discreto} starting from $E$, uniform in $\psi$ and $h$, holding in the exterior of $E$. We remark that, even if in the reference the anisotropy $\phi$ considered did not depend on $x$, all the arguments hold with minor modifications also in our case. We recall the proof of this result, as similar techniques will be used later on.
\begin{lemma}\label{lemma 3.7 in ChaMorNovPon}
    Let $E$ be a bounded, closed set, $h>0$ , and $g\in L^\infty(\R^N)$. Let $E'$ be a minimizer of
    \begin{equation*}
        P_\phi(F )+ \int_{F }  \dfrac {\sd^\psi_E}h  + g .
    \end{equation*} 
    Then, there exists $\sigma>0$, depending on $\lambda$, and $r_0\in (0,1)$, depending only on $N,\lambda,G:=\|g\|_{L^\infty(F)}$, with the following property: if $\bar x$ is such that $|E'\cap B_s(\bar x)|>0$ for all $s>0$ and $B_r(\bar x)\cap E=\emptyset$ with $r\le r_0$, then
    \begin{equation}\label{density estimates lemma 3.7 ChaMorNovPon}
        |E'\cap B_r(\bar x)|\ge \sigma r^N.
    \end{equation}
    Analogously, if $\bar x$ is such that $| B_s(\bar x)\setminus E'|>0$ for all $s>0$ and $B_r(\bar x)\subseteq E$ with $r\le r_0$, then
    \begin{equation*}
        | B_r(\bar x)\setminus E'|\ge \sigma r^N.
    \end{equation*}
\end{lemma}

\begin{proof}
    For all $s\in(0,r)$, set $E'(s):=E'\setminus B_s(\bar x)$. Note that, for a.e. $s$ we have
    \[  P_\phi(E'(s))=P_\phi(E')- P_\phi(E'\cap B_s(\bar x))+ \int_{E'\cap \bd B_s(\bar x)} \left(\phi(x,\nu(x)) + \phi(x,-\nu(x))\right) \udH(x),\]
    where $\nu$ denotes the outer normal vector of the set $E'\cap \bd B_s(\bar x)$. Since  { $E'\cap  B_s(\bar x)\subset E^c$ and $\sd^\psi_E\ge0$ in $E^c$, one has  $\int_{E'\cap B_s(\bar x)}\sd^\psi_E \ge 0,$ and therefore}  the minimality of $E'$ implies
    \[ P_\phi(E'\cap B_s(\bar x)) + \int_{E'\cap B_s(\bar x)} g \le \int_{E'\cap \bd B_s(\bar x)} \left(\phi(x,\nu(x)) + \phi(x,-\nu(x))\right)  \udH(x). \]
    By the bound on the $\phi-$perimeter and using the classical isoperimetric inequality (whose constant is denoted $C_N$) we obtain
    \begin{align*}
        2\lambda \mathcal H^{N-1}(E'\cap \bd B_s(\bar x))&\ge \frac 1\lambda P(E'\cap B_s(\bar x))+ \int_{E'\cap B_s(\bar x)} g\\
        &\ge \frac 1\lambda C_N|E'\cap B_s(\bar x)|^{\frac{N-1}N}-\|g\|_\infty |E'\cap B_s(\bar x)|\ge \dfrac{C_N}{2\lambda}|E'\cap B_s(\bar x)|^{\frac{N-1}N},
    \end{align*}
    provided $|E'\cap B_s(\bar x)|^{1/N}\le C_N/(2\lambda\|g\|_\infty)$, which is true if $r_0$ is small enough. Since the \textit{rhs} is positive for every $s$, we conclude
    \begin{equation}\label{eq 3.19 ChaMorNovPon}
        \dfrac \ud {\ud s} |E'\cap B_s(\bar x)|^{\frac{1}N}\ge \dfrac{C_N}{4\lambda^2 N}\quad \text{for a.e. }s\in(0,r).
    \end{equation}
    The thesis follows by integrating the above differential inequality. The other case is analogous.
\end{proof}

\begin{oss}
    Requiring that the anisotropy $\psi$ is bounded uniformly from above and below ensures that the results of the previous Lemmas \ref{lemma density estimate} and \ref{lemma 3.7 in ChaMorNovPon} can be read in terms of the $\psi-$balls. For example, for any $r\ge 0$ and $x\in\R^N$, equation \eqref{density estimates lemma 3.7 ChaMorNovPon} could be read as $ |E'\cap B^\psi_r(\bar x)|\ge \sigma  c_\psi^{-N} r^N,$    provided $\bar x$ is such that $|E'\cap B^\psi_s(\bar x)|>0$ for all $s>0$ and $B_r^\psi(\bar x)\cap E=\emptyset$, and holds for all $r\le r_0/c_\psi$. Here, $\sigma$ is as in Lemma \ref{lemma 3.7 in ChaMorNovPon} and depends only on $\lambda$. Analogous statements holds for Lemma \ref{lemma a priori estimate}.
\end{oss}

We now provide some estimates on the evolution of balls under the discrete flow. We start by a simple remark concerning the boundedness of the evolving sets.

\begin{oss}\label{evolution bounded sets}
    A simple estimate on the energies implies that the minimizers of \eqref{problema discreto} are bounded whenever $F$ is bounded. Indeed, assume $F\subseteq B_R$   and consider $B_\rho(x)\cap  (E\setminus B_R)\neq \emptyset $: testing the minimality of $E$ against $F$ we easily deduce
    $$ \dfrac R{2h}|B_\rho(x)\cap E|\le \int_{E\cap B_\rho(x)}\frac {\sd^\psi_F}h\le P_\phi(F)  + \|F_h(\cdot,t)\|_\infty |E\triangle F|\le P_\phi(F)  + \|f\|_\infty   (|F|+|E|) . $$
    Employing the density estimates of  {Lemma} \ref{lemma 3.7 in ChaMorNovPon} and sending $R\to\infty$, we get a contradiction, as the isoperimetric inequality implies that $|E|$ is bounded since $\mathscr F_{h,t}^F(F)<\infty.$
\end{oss}                                                                                             
We now want to prove finer estimates on the speed of evolution of balls. These estimates are classically a crucial step in order to prove existence of the flow. In the case under study,   the main difficulties come from the inhomogeneity of the functionals considered, as in the homogeneous case convexity arguments easily yield the boundedness result, for example. We  will use a  \virg{variational} approach in the spirit of \cite{ChaMorPon15} (but see also \cite[Lemma 3.8]{MugSeiSpa} for a different proof relying more on the smoothness of the evolving set).

\begin{lemma}	\label{lemma estimates on balls}
    For every $R_0>0$ there exist $h_0(R_0)>0$ and $C(R_0,\phi,\psi,f)>0$ with the following property: For all $R\ge R_0$, $h\in(0,h_0)$, $t>0$ and $x\in\R^N$ one has
    \begin{equation}\label{evolution law ball}
        T_{h,t}(B_R(x))\supset B_{R-Ch}(x).
    \end{equation}
\end{lemma}

\begin{proof}
    We divide the proof into three steps. In the following, the constants  $\sigma,r_0$ are those of Lemma~\ref{lemma 3.7 in ChaMorNovPon}. We will assume $x=0$ for simplicity. We fix $R\ge R_0$ and denote $E:=T_{h,t} B_R$.

    \noindent\textbf{Step 1.} We prove that, given $a\in (0,\sigma),\varepsilon\in(0,1)$, we can ensure $|B_{R(1-\varepsilon)}\setminus E|<  a\,R^N (1-\varepsilon)^N$ for $h$ small enough. Indeed, assume by contradiction $|B_{R(1-\varepsilon)}\setminus E| \ge  a\,R^N (1-\varepsilon)^N$. Testing the minimality of $ E$ against $B_R$, we obtain 
    \[   \int_{(B_{R(1-\varepsilon)}\setminus E)\cup (E\setminus B_{R})} \dfrac {|\sd^\psi_{B_R}|}h\le  \dfrac 1h \int_{B_{R}\triangle E} |\sd^\psi_{B_R}| \le P_\phi(B_R) {-\int_{B_R\setminus E}F_h+\int_{E\setminus B_R}  F_h} ,  \]
    and estimating $|\sd^\psi_{ B_R}|\ge  {R}\varepsilon/c_\psi$  on $B_{R(1-\varepsilon)}\setminus E$, we get
    \begin{align*}
        \dfrac{  {R}\varepsilon}{h c_\psi} |B_{R(1-\varepsilon)}\setminus E|\le P_\phi(B_R)+ \|f\|_\infty \left(\omega_N  R^N+ |B_{R(1+\varepsilon)}\setminus B_R|\right)+\int_{E \setminus B_{R(1+\varepsilon)}}\left(F_h-\dfrac{|\sd^\psi_{ B_R}|}h\right) .
    \end{align*}
    Taking $h\le \varepsilon /(c_\psi\|f\|_\infty) $, the last term on the \textit{rhs} is negative, thus 
    \[ \dfrac{  {R}\varepsilon}{h c_\psi}  |B_{R(1-\varepsilon)}\setminus E|\le \P(B_R)+\|f\|_\infty R^N(\omega_N+2^{N+1}\varepsilon). \]
    We employ the  hypothesis  to obtain
    \[ \dfrac a {h c_\psi}  \, \varepsilon (1-\varepsilon)^N R^{N+1}\le c_\psi N\omega_N R^{N-1}+cR^N, \]
    a contradiction for $h\le c a\,\varepsilon \,  (1-\varepsilon)^N \min\{1, {R^2}\},$ where $c$ is a constant depending on $N,\phi,\psi,\|f\|_\infty.$

    \noindent\textbf{Step 2.}   {Using Step 1, we prove that $B_{R/2}\subset E $ for $h$ small. Assume that $R\le r_0$: by following}   
    the second part of the proof of Lemma \ref{lemma 3.7 in ChaMorNovPon}  we obtain equation \eqref{eq 3.19 ChaMorNovPon}, which reads 
    \[\dfrac \ud {\ud s} |B_s\setminus E|^{1/N}\ge \dfrac{C_N}{4\lambda^2 N}=\sigma^{1/N}\quad \text{for a.e }s\in (0,R) .   \]
     {Applying the previous step with $\varepsilon=1/4, a=\sigma/3^N$, it holds $|B_{3R/4}\setminus E|\le \sigma R^N/4^N$ for all $h\le c(N,\phi,\psi,f)R$. Therefore, one}
    deduces the existence of a positive extinction radius 
    \begin{equation}\label{extinction radius}
        R^*= \dfrac{3R}4 - \dfrac{|B_{3R/4}\setminus E|^{1/N}}{\sigma^{1/N}}\ge \frac R2
    \end{equation}
    such that $|B_{R^*}\setminus E|=0$, which proves the claim.
    Clearly, taking $h\le cR_0$ the smallness assumption on $h$ is uniform for $R\ge R_0$.
    
    If $R\ge r_0$ one simply uses a covering argument. For any $x\in B_{R-r_0}$, applying the previous result to the ball $B_{r_0}(x)$ and using the comparison principle of Lemma \ref{comparison principle}, we conclude that $\forall h\le c\, r_0$ it holds
    \[ \bigcup_{x\in B_{R-r_0}} B_{r_0/2}(x) {\subset\joinrel\subset} E. \]
    
    \noindent\textbf{Step 3.} We conclude the proof. By the previous two steps  {and Remark \ref{evolution bounded sets},} taking $h$ small enough, we see that
    \[\rho :=\sup\{  r>0 \ :\ |B_r\setminus E|=0 \} \in  {(R/2,+\infty)}. \]
    We can assume $\rho\le R,$ otherwise the result of the lemma  is trivial. Consider the  vector field $\nabla_p \phi\left(x,\frac x{|x|}\right)\in C^1(\R^N,\R^N)$. Then, recalling \eqref{P via calibration}, we get $ \P( {G})\ge  -\int_{\R^N} D\chi_{G} \cdot \nabla_p\phi(x,x/|x|)$ for all $G$ set of finite perimeter and
    \[ \P((1+\varepsilon)B_\rho)= \int_{\R^N} D\chi_{(1+\varepsilon)B_\rho} \cdot\left(  -\nabla_p\phi\left(x,\frac x{|x|}\right) \right). \]
    Setting $W_\varepsilon=(1+\varepsilon)B_\rho\setminus E$, by submodularity on $(1+\varepsilon)B_\rho,E $ and exploiting the minimality of $E$, we obtain
    \begin{align*}
         \int_{\R^N} \nabla_p\phi\left(x,\dfrac x{|x|}\right) \cdot  D\chi_{W_\varepsilon}&=\int_{\R^N} \nabla_p \phi\left(x,\frac x{|x|}\right) \cdot \left( D\chi_{(1+\varepsilon)B_\rho}- D\chi_{(1+\varepsilon)B_\rho\cap E} \right)\\
        &\le \P((1+\varepsilon)B_\rho\cap E)-\P((1+\varepsilon)B_\rho)\\
        &\le \P(E)-\P((1+\varepsilon)B_\rho \cup E)\\
        &\le \dfrac 1h \int_{W^\varepsilon} \sd^\psi_{B_R} {-}\int_{W_\varepsilon}F_h(x,t)\ud x.
    \end{align*}
    We conclude, using the divergence theorem ,
    \[ \int_{W^\varepsilon} -\div \nabla_p \phi\left(x,\frac x{|x|}\right) \le  \dfrac 1h \int_{W^\varepsilon} \sd^\psi_{B_R}+\|f\|_\infty|W_\varepsilon|.\]
    Dividing by $|W^\varepsilon|$ and sending $\varepsilon\to 0$ we obtain 
    \[ \fint_{\bd B_\rho\cap E}  -\div \nabla_p \phi\left(x,\frac x{|x|}\right) {\udH} \le\frac 1{c_\psi} \dfrac{\rho-R}h+\|f\|_\infty.\]
    Exploiting the regularity assumptions  on $\phi$, we remark that
    $$|\div \nabla_p \phi| = |\text{tr}\left(  \nabla_x\nabla_p \phi + \nabla^2_p\phi \nabla(x/|x|) \right)| \le C\left( 1+\frac 1{|x|}\right).$$ 
    Thus, we obtain
    \[  -C\left( 1+\dfrac 1\rho \right) \le \dfrac{\rho-R}h,\]
    which implies that $\rho\in (0,\rho_1)\cup (\rho_2, R)$ for $\rho_{1,2}=\left( R-Ch\mp\sqrt{(R-Ch)^2-4Ch}\right)/2,$ as long as $h\le R_0^2/(4C)$.  Since the choice $\rho\le \rho_1< R/2$ is not admissible, we conclude the proof by estimating
    \begin{equation*}
        \rho_2=R-Ch + \dfrac{R-Ch}2\left( \sqrt{1-\dfrac{4Ch}{(R-Ch)^2}} -1  \right)	\ge R-Ch-\dfrac{Ch}{R-Ch},
    \end{equation*}
    from which the thesis follows.
\end{proof}

The proof of the previous result can be employed to prove an estimate from above of the evolution speed of the flow, as the following result shows. Since the proof follows the same lines and is easier in this case, we only sketch it.

\begin{lemma}	\label{lemma a priori estimate}
    Fix $T>0$ and $R_0>0$. Then, there exist positive constants $C=C(\phi,\psi, f,R_0)$ and $ h_0=h_0(R_0)$ such that, for every $R\ge R_0$ and $h\le h_0$, if $E_0\subseteq B_{R}$, then $E_t^{(h)}\subseteq B_{R+CT}$ for all $t\in (0,T)$.
\end{lemma}
\begin{proof}
    Choose $h$ small as in the previous result and set
    \[\rho =\inf\{  r>0 \ :\ |E\setminus B_r |=0 \}\in  {(R/2,+\infty)}.   \]
    We can assume $\rho\ge R,$ otherwise the result is trivial.  Defining $W_\varepsilon=E\setminus (1-\varepsilon)B_\rho$  and reasoning as before we obtain
    \begin{align*}
         \int_{\R^N} \nabla_p\phi\left(x,\dfrac x{|x|}\right)  \cdot  D\chi_{W^\varepsilon}&=\int_{\R^N} \nabla_p\phi\left(x,\dfrac x{|x|}\right)  \cdot \left(  D\chi_{(1-\varepsilon)B_\rho\cup E} -D\chi_{(1-\varepsilon)B_\rho}\right)\\
        &\ge -\P((1-\varepsilon)B_\rho\cup E)+\P((1-\varepsilon)B_\rho)\\
        &\ge -\P(E)+\P((1-\varepsilon)B_\rho \cap E)\\
        &\ge \dfrac 1h \int_{W^\varepsilon} \sd^\psi_{B_R} {-}\int_{W_\varepsilon}F_h(x,t)\ud x.
    \end{align*}
    As in the previous proof, we arrive at
    \[   \dfrac{\rho-R}h\le C\left( 1+\dfrac 1\rho \right) ,\]
    which implies that $\rho\le \rho_2=\left( R+Ch + \sqrt{(R+Ch)^2+4Ch}\right)/2\le R+Ch,$ up to changing $C$.  
\end{proof}

\subsection{Existence of flat flows}

In the following, we will prove that the discrete flow (defined in \eqref{def discrete flow}) defines a discrete-in-time approximation of a weak solution to the mean curvature flow, which is usually known as a \virg{flat} flow  (because the approximating surfaces $\bd^* E_t\h$ converge in the \virg{flat} distance of Whitney to the limit $\partial^*E_t$, see \cite{AlmTayWan}).

We start by proving uniform bounds on the distance between two consecutive sets of the discrete flow and on the symmetric difference between them. We introduce the time-discrete normal velocity: for all $t\ge 0$ and $x\in\R^N$, we set 
\[ v_h(x,t):=\begin{cases}
    \frac 1h \sd^\psi_{E_{t-h}^{(h)}}(x) \quad&\text{for }t\in [h,+\infty)\\
    0 &\text{for }t\in [0,h).
\end{cases}\]
The following result provides a bound on the $L^\infty-$norm of the discrete velocity. Since the proof is essentially the same of \cite[Lemma 2.1]{LucStu}, we will omit it. The only difference is that we use the upper and lower bounds of \eqref{inclusione psi palle} to work with Euclidean balls.
\begin{lemma}\label{L infty estimate}
    There exists a positive constant $c_\infty$ depending only on $N,\psi$ with the following property. Let $E_0$ be a bounded set of finite perimeter and let $\{ E_t^{(h)} \}_{t\in (0,T)}$ be a discrete flow starting from $E_0$. Then, 
    \begin{equation*}
        \sup_{ E_t^{(h)}\triangle  E_{t-h}^{(h)}} |v_h(\cdot,t)|\le c_\infty h^{-1/2}
    \end{equation*}
     {for all $h$ sufficiently small.}
\end{lemma}

The following result can be found in \cite[Proposition 3.4]{MugSeiSpa} (see also \cite[Lemma 2.2]{FusJulMor}): it provides an estimate on the volume of the symmetric difference of two consecutive sets of the discrete flow. The proof is analogous to the one in the reference.

\begin{lemma}    
     {There  exists a constant $C$  such that    for every $t\ge h$ the discrete flow $E_t\h$ satisfies for all $h$ sufficiently small
    \begin{equation}\label{stima L1}
        |E_{t+h}\h \triangle E_t\h|\le C\left( lP_\phi(E_t\h)+\dfrac 1l \int_{E_t\h \triangle E_{t+h}\h}|\sd^\psi_{E_t\h}| \right) \quad \forall l\le c \sqrt h,
    \end{equation}
    where $c$ is  a positive constant depending on $N,\psi$.}
\end{lemma}

We are now able to prove an uniform bound on  the perimeter of the evolving sets. The proof follows \cite[Proposition 2.3]{FusJulMor}.
\begin{lemma} \label{lemma 1.1 LucStu}
    For any initial bounded set $E_0$ of finite $\phi-$perimeter and $h$ small enough, the discrete flow $\{E_t^{(h)}\}$  satisfies
    \[ P_\phi(E_t^{(h)})\le C_T \quad \forall t\in (0,T),\]
    for a suitable constant $C_T=C_T(T,E_0,f,\phi,\psi)$. 
\end{lemma}

\begin{proof}
    By testing the minimality of  $E^{(h)}_t$  against  $E^{(h)}_{t-h}$  we obtain $\forall t\in [h,T)$
    \begin{equation}\label{eq 2.4 FJM}
        P_\phi( E^{(h)}_t )+  {\frac 1h} \int_{E_t\h\triangle E_{t-h}\h}|\sd^\psi_{E_{t-h}\h}|\le P_\phi(E^{(h)}_{t-h})+\|f\|_{\infty}|E_t\h\triangle E_{t-h}\h|.
    \end{equation}
    Combining this estimate with \eqref{stima L1} for $l=2C h\|f\|_\infty\ll  \sqrt h$, where $C$ is the constant  {appearing in equation \eqref{stima L1}}, we obtain  {for $h$ sufficiently small} 
    \begin{equation}	\label{eq iterativa}
        P_\phi( E^{(h)}_t )+\dfrac1{ {2h}}\int_{E_t\h\triangle E_{t-h}\h}|\sd^\psi_{E_{t-h}\h}|\le \left( 1+  {2C^2} h\|f\|_\infty^2 \right) P_\phi(E^{(h)}_{t-h})
    \end{equation}
    Iterating the previous estimate, we find
    \begin{align*}
        P_\phi( E^{(h)}_t )&\le  (1+ {2C^2\|f\|_{\infty}}h)^{\left[ \frac t h \right]-1} \P(E_h\h).
    \end{align*}
     {In order to estimate $\P(E\h_h)$  we start by observing that  Remark \ref{evolution bounded sets}, for $h=h(E_0)$ small enough, implies $E_h\h\subseteq B_{2r}$, where $E_0\subseteq B_r$. Therefore, by \eqref{eq 2.4 FJM} for $t=h$ we obtain  $\P(E_h\h)\le \P(E_0) + c$ and we conclude $\P(E_t\h)\le C_T(\P(E_0)+1)$.}
\end{proof}

 {We then present a sketch of the proof of the local H\"older continuity in time of the discrete flow, uniformly in $h$, which can be  deduced as in \cite[Proposition 2.3]{FusJulMor}. We highlight the main differences.}
\begin{prop}\label{prop 3.3.1 MugSeiSpa}
    Let $E_0$ be an initial bounded set of finite $\phi-$perimeter and $T>0$. Then, for $h$ small enough, for a discrete flow $\{ E\h_t \}$ starting from $E_0$ it holds
    \[  |E_{t}^{(h)}\triangle E_{s}^{(h)}|\le C_T|t-s|^{1/2}\quad \forall  h\le t\le s < T,  \]
    for a suitable constant  $C_T=C_T(T,E_0,f,\phi,\psi)$.  
\end{prop}

\begin{proof}  
Following the previous proof, employing again \eqref{eq iterativa} we find 
\begin{align*}
    \P( {E_{2h}\h})&+ {\frac 12}\int_{E_{ {2h}}\h\triangle E_ {h}\h}|v_h {(\cdot,2h)}|+ {\frac 12}\int_{E_ {h}\h\triangle E_0\h}|v_h {(\cdot,h)}|\\
    & \le (1+ch)\P(E_ {h}\h)+ {\frac 12}\int_{E_ {h}\h\triangle E_0\h}|v_h {(\cdot,h)}|\\
    &\le (1+ch)\left(\P(E_ {h}\h)+\int_{E_ {h}\h\triangle E_0\h}|v_h| {(\cdot,h)}\right)\le (1+ch)^2 \P(E_0).
\end{align*}
Iterating, we conclude as before
\begin{equation}\label{eq 3.18 MugSeiSpa}
    \sum_{k=1}^{[T/h]}\int_{E_{kh}^{(h)}\triangle E_{(k-1)h}^{(h)}} |v_h {(\cdot,kh)}|   \le  C_T(\P(E_0)+1).
\end{equation}
Therefore, combining the previous results and applying \eqref{stima L1} with $l=h\ll \sqrt h$, we obtain 
\begin{equation}\label{eq 3.19 MugSeiSpa}
    \int_h^T |E_{t}^{(h)}\triangle E_{t-h}^{(h)}|\le c\sum_{k=1}^{[T/h]}   \left( h P_\phi(E_{kh}^{(h)}) +\int_{E_{kh}^{(h)}\triangle E_{(k-1)h}^{(h)}} |v_h {(\cdot,kh)}|   \right)  \le C_T\left( P_\phi(E_0)+1\right).
\end{equation}
 {The proof then follows the one of \cite[Proposition 2.3]{FusJulMor}, from equation $(2.5)$ onward.}
\end{proof}

We finally prove the main result of this section, the existence of flat flows.
\begin{proof}[Proof of Theorem \ref{theorem existence flat flow}]  The proof is classical and we only sketch it. By the uniform equicontinuity of the approximating sequence of Proposition \ref{prop 3.3.1 MugSeiSpa} and compactness of sets of finite perimeter (by Lemma \ref{lemma a priori estimate} and  \ref{lemma 1.1 LucStu}) we can use the Ascoli-Arzelà theorem to prove that the sequence  $(E_{t}^{(h_k)})_{k\in\N}$  converges in $L^1$ to sets $E_t$ for all times $t\ge 0$ and that the family $\{E_t \}_{t\ge 0}$ satisfies the $1/2-$H\"older continuity property, locally uniformly in time. The other property is then easily deduced.
\end{proof}

\subsection{Existence of distributional solutions}
From Theorem \ref{theorem existence flat flow} we deduce the existence of a subsequence $(h_k)_{k\ge 0}$ such that
\begin{equation}
    D \chi_{E_{t}^{(h_k)}}  \weakstar D \chi_{E_t}\qquad\forall t\ge0.
    \label{convergenza senso misure}
\end{equation}
We will also assume \eqref{hp per}, remarking that it implies
\begin{equation}\label{convergenza perimeteri}
    \lim_{k\to \infty } \P(E_{t}^{(h_k)})=  P_\phi(E_t)\qquad \text{for a.e. }t\in [0,+\infty).
\end{equation}

Our aim is to derive \eqref{legge curvatura} and \eqref{legge velocita} from the Euler-Lagrange equation \eqref{EL equation} and passing to the limit $h\to 0$. To achieve  {this},  we will prove that the discrete velocity is a good approximation (up to multiplicative factors) of the discrete evolution speed of the sets.  Notice that \eqref{legge curvatura} is a weak formulation of \eqref{smooth law}, while \eqref{legge velocita} establishes the link between $v$ and the velocity of the boundaries of $E_t$. Indeed, law \eqref{smooth law} can be interpreted as looking for a family $\{E_t\}_{t\ge0} $ of sets, whose normal vector $\nu_{E_t}$ and $\phi-$curvature $H^\phi_{E_t}$ are well-defined objects and a function $v\ :\ [0,\infty)\times \R^N\to \R$ such that  {for every $t\in [0,+\infty)$ and $x\in\bd E_t$}
\begin{equation}
    \begin{cases}
        v(x,t)&=-H^\phi_{E_t}(x)+f(x,t)\\
        V(x,t)&=\psi(x,\nu_{E_t}(x))v(x,t),
    \end{cases}
\end{equation}
where $V$ represents the normal velocity of evolution, obtained as  {the} limit as $h\to 0$ (in a suitable sense) of the ratio
\[ \frac {\chi_{E_t}-\chi_{E_{t-h}}}h. \]

In this whole section we will assume that hypothesis \eqref{hp dim} holds. In particular, the sets defining the discrete flow are smooth hypersurfaces in $\R^N$. Moreover, we require hypotheses \eqref{hp lip psi 1} to hold.

We start by estimating in time the $L^2-$norm of the discrete velocity. 	The proof is the same  {as the one presented in} \cite[Lemma 3.6]{MugSeiSpa}, up to using the density estimates on the $\phi-$perimeter of Lemma \ref{lemma density estimate} and considering the $\psi-$balls instead of the Euclidean one.
\begin{prop}\label{L2 bound velocity}
        Let $\{E_{t}^{(h)}\}_{t\ge 0}$ be a discrete  flow starting from an initial bounded set $E_0$ of finite $\phi-$perimeter. Then, for any $T>0$ and for $h$ small enough, it holds
    \begin{equation*}
        \int_0^T \int_{\bd E_{t}^{(h)}} v_h^2\udH\ud t\le C_T,
    \end{equation*}
    for a suitable constant $C_T=C_T(T,E_0,\phi,\psi,f)$. 
\end{prop}
Recalling now the Euler-Lagrange equation \eqref{EL equation} and Lemma \ref{lemma 1.1 LucStu} we conclude
\begin{equation}\label{L2 bound curvatures}
    \int_0^T\int_{\bd E_{t}^{(h)}} \left( H^\phi_{E_{t}^{(h)}} \right)^2= \int_0^T\int_{\bd E_{t}^{(h)}} \left( v_h-F_h \right)^2 \le C_T,
\end{equation}

We now prove an estimate on the error between the discrete velocity $\psi(\cdot,\nu_{E_t})v_h(\cdot,t)$ and the discrete time derivative of $\chi_h$. The proof of this result is based on a double blow-up argument, and the smoothness of  sets (locally) minimizing the $\phi-$perimeter is essential.  We will split the proof in various lemmas: the first one concerns the composition of blow-ups,  {and is a well-known result to the experts. We present a simple proof since we could not find a reference}.
\begin{lemma}[Composition of blow-ups] \label{lemma composition of blow-ups}
    Consider $0<\beta<\beta'<1$.	Assume that  {$(A_h)_{h\in [0,1]}$ is a family of measurable sets} such that the following blow-ups converge as $h\to 0$
    \begin{align*}
        \dfrac{ {A_h-x_h}}{h^\beta}\to A_1&\quad \text{in $L^1_{loc}$}\\[1ex]
        h^{-(\beta'-\beta)}A_1\to A_2&\quad \text{in $L^1_{loc}$},
    \end{align*}
    where  {$x_h\in\bd A_h$ for all $h\in[0,1]$}. Then, if the composition of the blow-ups $h^{-\beta'} {(A_h-x_h)}$ converges  in $L^1_{loc}$, the limit coincides with $A_2$.
\end{lemma}
\begin{proof}
    We can assume \textit{wlog} $ {x_h}=0$. Denote with $A_3=L^1_{loc}-\lim_{h\to 0} h^{-\beta'} {A_h}$. We fix a ball $B_M$ and $\varepsilon>0$. There exists $h^*$ such that $\forall h\le h^*$ it holds 
    \[  |((h^{-\beta'} {A_h})\triangle A_3)\cap B_M|\le \varepsilon,\quad |((h^{-\beta'+\beta}A_1)\triangle A_2)\cap B_M| \le \varepsilon.\]
    We fix $h$ and \textit{wlog} assume $Mh^{\beta'-\beta}\le 1$. 	Taking $\tilde h<h$ suitably small (depending on $h,\varepsilon$), we can ensure 
    \[ |((\tilde h^{-\beta} {A_h})\triangle A_1)\cap B_1|\le \varepsilon h^{N(\beta'-\beta)}. \]
    Since $\tilde h^{-\beta} h^{-(\beta'-\beta)}>h^{-\beta'}$, there exists $\bar h< h$ such that $\bar h^{-\beta'}=\tilde h^{-\beta} h^{-(\beta'-\beta)}.$ We can then estimate 
    \begin{align*}
        |(A_3\triangle A_2)\cap B_M|&\le |(A_3\triangle \bar h^{-\beta'}  {A_h})\cap B_M| + |((h^{-\beta'+\beta}) A_1 \triangle (\bar h^{-\beta'}  {A_h}))\cap B_M|  \\
        &\ + |((h^{-\beta'+\beta} A_1)  \triangle A_2)\cap B_M|\\
        &\le 2\varepsilon + h^{-N(\beta'-\beta)}|(A_1 \triangle (\tilde  h^{-\beta}  {A_h}))\cap B_{Mh^{\beta'-\beta}}|\\
        &\le 2\varepsilon + h^{-N(\beta'-\beta)}|(A_1 \triangle (\tilde  h^{-\beta}  {A_h}))\cap B_{1}|\le 3\varepsilon.
    \end{align*}
\end{proof}

We now compute some estimates on the normal vector on the boundary of the evolving sets,  {following the proof of \cite[Lemma 4.2]{MugSeiSpa} (see also \cite[Proposition 2.2]{LucStu})}.  {We fix $c_\infty$ as the constant appearing in Lemma \ref{L infty estimate}.}

 {In the sequel, we will denote by $\omega(h) $ a  modulus of continuity, that is a continuous increasing function $\omega:[0,1]\to \R$ with $\omega(0)=0$, which can possibly change from statement to statement and line to line to absorb constants independent of $h$.}

\begin{lemma}\label{lemma computations}
    Assume \eqref{standing hp} and \eqref{hp lip psi 1}. For given constants $1/2<\beta'<\alpha<1$ and $T>2$, there exists a  {modulus of continuity $\omega$} with the following property. Consider $t\in[2h,T]$ and $ {x_h}\in\bd E_t^{(h)}$ such that 
    \begin{equation}\label{eq 4.6 MugSeiSpa}
        |v_h(t,y)|\le h^{\alpha-1}\quad \forall y\in B_{c_\infty\sqrt h}( {x_h})\cap (E_t^{(h)}\triangle E_{t-h}^{(h)}).
    \end{equation} 
    Then, there exists $\nu\in  {S^{N-1}}$ such that 
    \begin{align}
        |\nu_{E_t^{(h)}}(\cdot)-\nu|&\le \omega(h)\quad \text{in }B_{h^{\beta'}}( {x_h})\cap \bd E_t^{(h)}\nonumber\\
        |\nu_{E_{t-h}^{(h)}}(\cdot)-\nu|&\le \omega(h)\quad \text{in }B_{h^{\beta'}}( {x_h})\cap \bd E_{t-h}^{(h)}\label{eq 4.8 MugSeiSpa}.
    \end{align}
\end{lemma}

\begin{proof}
    We fix $\frac12<\beta<\beta'<\alpha$ and $0<R<h^{\frac 12-\beta}/c_\psi$. Testing the minimality of $E^{(h)}_s, s=t,t-h$, we find 
    \begin{equation}\label{eq 4.9 MugSeiSpa}
        P_\phi(E_s^{(h)}, B_{Rh^\beta}( {x_h}))\le P_\phi(G, B_{Rh^\beta}( {x_h}))+\dfrac1h\int_{G\triangle E_s^{(h)}}|\sd^\psi_{E_{s-h}^{(h)}}|+\int_{G\triangle E_s\h}|F_h|,
    \end{equation}
    for any set $G$ of finite perimeter such that $G\triangle E_s^{(h)} {\subset\joinrel\subset} B_{Rh^\beta}( {x_h})$. Using Lemma \ref{L infty estimate}, the $1-$Lipschitz regularity of $\sd^\psi$ and \eqref{eq 4.6 MugSeiSpa}, we deduce $|v_h(s,y)|\le c_\psi Rh^{\beta-1}+c_\infty h^{-1/2}\le (1+c_\infty)h^{-1/2}$ for any $y\in B_{Rh^\beta}( {x_h})\cap (E^{(h)}_s\triangle F)$. Plugging this inequality in \eqref{eq 4.9 MugSeiSpa}, we find 
    \begin{equation}\label{eq 4.10 MugSeiSpa}
        P_\phi(E_s^{(h)}, 	B_{Rh^\beta}( {x_h}))\le P_\phi(G, B_{Rh^\beta}( {x_h}))+\dfrac{1+c}{\sqrt h}|F\triangle E_s^{(h)}|+\|f\|_{\infty} |G\triangle E_s^{(h)}|.
    \end{equation}
    We then introduce the blown-up sets for $s=t,t-h$, defined as
    \begin{align*}
        E_s^{(h),\beta}:=h^{-\beta}\left(E\h_s- {x_h}\right).
    \end{align*}
    Rescaling equation \eqref{eq 4.10 MugSeiSpa}, we easily find that $E_s^{(h),\beta}$ is a $(\Lambda_h,r_h)-$minimizer of the $\phi( {x_h}+h^\beta\cdot,\cdot)-$perimeter, with $\Lambda_h=(1+c)h^{\beta-1/2},\  r_h=h^{1/2-\beta}$. Moreover, scaling the density estimates \eqref{perimeter density estimates, h dip} we have a uniform bound on the perimeters of the sets $E_s^{(h),\beta}$ in each ball $B_R$. By compactness, there exist two sets $E_1^\beta, E_2^\beta$ such that 
    \[ E_t^{(h),\beta}\to E_1^\beta,\ E_{t-h}^{(h),\beta}\to E_2^\beta\qquad \text{in }L^1_{loc}.\]
    Then, by scaling and \eqref{eq 4.6 MugSeiSpa} we find 
    \[  |\sd^\psi_{E_{t-h}^{(h),\beta}}(\cdot)|\le  {c_\infty}h^{\alpha-\beta} \qquad \text{on } B_{h^{1/2-\beta}}(0)\cap (E_t^{(h),\beta}\triangle E_{t-h}^{(h),\beta}),	 \]
    thus we easily conclude that $E^\beta:=E_1^\beta= E^\beta_2.$  {By Lemma \ref{lemma a priori estimate} we can assume that $x_h\to x_0$ as $h\to 0$, up to subsequences.} Moreover, by closeness of $\Lambda_h-$minimizers under $L^1_{loc}-$convergence (see e.g. \cite[Theorem 2.9]{DePMag1}), one can see that  $E^\beta$ is a   {0-minimizer for the $\phi(x_0,\cdot)-$perimeter}. Thus, by complete regularity, it is a smooth $C^2$ set. We can then employ the classic  blow-up theorem  to deduce that, for a fixed $\beta'\in (\beta,\alpha)$, the blow-up $h^{-(\beta'-\beta)}E^\beta$ converges to a half-space $\mathbb H=\{ x\cdot \nu\le 0 \}$ as $h\to 0$. Finally, the blow-ups
    $$ E_s^{(h),\beta'}:=\dfrac{E^{(h)}_s- {x_h}}{h^{\beta'}}$$
    admit a converging subsequence by compactness of sets of finite perimeter and by rescaling equation \eqref{eq 4.10 MugSeiSpa}. Thus,  the previous Lemma \ref{lemma composition of blow-ups} implies
    \[ E_s^{(h),\beta'}\to \mathbb H\quad \text{in }L^1_{loc} \]
    as $h\to 0$. To conclude, the $\varepsilon-$regularity Theorem for $\Lambda-$minimizers (see e.g. \cite[Theorem 3.1]{DePMag1}) ensures that $E_s^{(h),\beta'}$ are uniformly $C^{1,\frac 12}$ sets in $B_1(0)$ for $s=t,t-h$ as $h\to 0$.
\end{proof}

We recall here an approximation result proved in \cite{LucStu} (see also \cite{MugSeiSpa} for a more detailed proof). We remark that the proof of this result is purely geometric and does not rely on the variational problem satisfied by the sets $E_t\h,E_{t-h}\h$.
\begin{cor*}[Corollary 4.3 in \cite{MugSeiSpa}]
    Under the hypotheses of Lemma \ref{lemma computations}, fix $0<\beta<\alpha$ and let $\textbf C_{h^{\beta}}$ be the open cylinder defined as
    \[ \textbf{C}_{h^{\beta}}(x_h,\nu):=\left\lbrace x\in\R^N : |(x-x_h)\cdot \nu|<\frac{h^\beta}2,\  \bigg|(x-x_h)-\left((x-x_h)\cdot \nu\right)\nu\bigg|<\frac{h^\beta}2\right\rbrace.\]
    Then, it holds
    \begin{align*}
        \Big|  \int_{\textbf{C}_{h^\beta/2}(x_h,\nu)} (\chi_{E_t^{(h)}}-\chi_{E_{t-h}^{(h)}})&\ud x   -    \int_{\bd E_t^{(h)}\cap \textbf{C}_{h^\beta/2}(x_h,\nu)} \sd_{E^{(h)}_{t-h}}  \udH \Big|\\
        &\le \omega(h)\int_{\textbf{C}_{h^\beta/2}(x_h,\nu)} |\chi_{E_t^{(h)}}-\chi_{E_{t-h}^{(h)}}|.
    \end{align*}
\end{cor*}
Carefully inspecting the proof, one indeed proves that there exists a geometric constant $C$ such that for any $ y\in B^{N-1}_{h^\beta/2}(x_h)$ 
\begin{equation}\label{computation sd}
    |\sd_{E_{t-h}\h}(y,f_t\h(y))\sqrt{1+|\nabla f_t\h(y)|^2}-\left( f_t\h(y)-f_{t-h}\h(y) \right)|\le \omega(h)|f_t\h(y)-f_{t-h}\h(y)|,
\end{equation}
where we set 
\[  \bd E_s\h\cap \textbf{C}=\{ (y,f_s\h(y))\in \R^{N-1}\times\R,\ |y|\le h^\beta/2 \}, \]
for $s=t,t-h$.

We briefly recall some classical results. Consider an anisotropy $\psi$, independent of the position. It is well-known that, for any  {closed } set $G\subseteq \R^N$,   setting $\sd^\psi_G$ as the distance from $G$ induced by $\psi^\circ$, then  the gradient of $\sd^\psi_G$ exists almost everywhere  and satisfies the  eikonal equation  {(for a proof see for instance \cite[Remark 2.2]{ChaNov22})}
\begin{equation}\label{eikonal eq}
    \psi(\nabla \sd^\psi_G)=1
\end{equation}
almost everywhere.
Moreover, in this particular case, in the definition of $\text{dist}^\psi$ we can consider just straight lines as follows from a simple application of Jensen's inequality: for any curve $\gamma$ as in the definition of $\dist^\psi,$ we have
\[ \int_0^1\psi^\circ(\dot \gamma(t))\ud t \ge \psi^\circ\left( \int_0^1\dot\gamma \right)= \psi^\circ(y-x). \]

\begin{prop}[Estimate on almost flat sets]\label{estimates on flat sets}
    Under the hypotheses of Lemma \ref{lemma computations}  {and with the same notation}, fix $\beta\in(0,\alpha)$ and let $\textbf C_{h^{\beta}}$ be the open cylinder defined as
    \[ \textbf{C}_{h^{\beta}}(x_h,\nu):=\left\lbrace x\in\R^N : |(x-x_h)\cdot \nu|<\frac{h^\beta}2,\  \bigg|(x-x_h)-\left((x-x_h)\cdot \nu\right)\nu\bigg|<\frac{h^\beta}2\right\rbrace.\]
    Then, it holds
    \begin{align*}
        \Big|  \int_{\textbf{C}_{h^\beta/2}(x_h,\nu)} (\chi_{E_t^{(h)}}-\chi_{E_{t-h}^{(h)}})&\ud x   -    \int_{\bd E_t^{(h)}\cap \textbf{C}_{h^\beta/2}(x_h,\nu)} \psi(x,\nu_{E_t\h})\,\sd^\psi_{E^{(h)}_{t-h}}  \udH \Big|\nonumber\\
        &\le \omega(h)\int_{\textbf{C}_{h^\beta/2}(x_h,\nu)} |\chi_{E_t^{(h)}}-\chi_{E_{t-h}^{(h)}}|.
    \end{align*}
\end{prop}

\begin{proof}
     {We recall that the modulus of continuity  $\omega$ may change from line to line to absorb constants independent of $h$.} 
    
    From the previous Lemma \ref{lemma computations} we know that, for $h$ suitably small, both $\bd E_t\h$ and $\bd E_{t-h}\h$ in $\textbf{C}_{h^\beta/2}(x_h,\nu)$ can  be written as graphs of functions of class $C^{1,\frac 12}$. Up to a change of coordinates, we can assume \textit{wlog} that $x_h=0, \nu=e_N$. For simplicity, we set $\textbf{ C}=\textbf{C}_{h^\beta/2}(0,e_N).$  We thus find
    \[  \bd E_s\h\cap \textbf{C}=\{ (y,f_s\h(y))\in \R^{N-1}\times\R,\ |y|\le h^\beta/2 \} \]
    for $s=t,t-h$, where $f_s\h : B_{h^\beta/2}^{N-1}\to \R$ are $C^{1,\frac12}$ functions with 
    \begin{equation*}
        \|\nabla f_s\h\|_{L^\infty(B_{h^\beta/2})}\le \omega(h).
    \end{equation*}
    We want to prove the following slightly stronger pointwise inequality: namely,  that for any point $x=(y,f_t\h(y))\in \bd E_t\h\cap \textbf{C}$, it holds
    \begin{equation}\label{estimate flatness}
        \left\lvert  \sd^\psi_{E_{t-h}^{(h)}}(x)\, \psi(x,\nu_{E_{t}^{(h)}}(x)) \sqrt{1+|\nabla f_t\h(y)|} - \left( f_t\h(y)-f_{t-h}\h(y) \right) \right\rvert \le \omega(h)|f_t\h(y)-f_{t-h}\h(y)|.
    \end{equation}
    Integrating the previous inequality over $\textbf{C}$ yields the thesis. Clearly, it is enough to prove \eqref{estimate flatness} at each point $x$ such that $|\sd^\psi_{E_{t-h}\h}(x)|>0$. We thus fix $x=(y,f_t\h(y))\in \bd E_t\h\cap \textbf{C}$ and denote by $x':=(y,f_{t-h}\h(y))$.  {We remark that these points depend on $h$, but we drop the subscript to ease notation}. It can be assumed without loss of generality that $x\notin E_{t-h}\h$, as the other case is analogous. \\
    \noindent {\textbf{Step 1}  We now prove that, with the notation previously introduced, it holds
    \begin{equation}\label{est dist}
        |\sd'_{E_{t-h}\h}(x)-\sd^\psi_{E_{t-h}\h}(x)|\le \omega(h) |f_t\h(y)-f_{t-h}\h(y)|,
    \end{equation}
    where $\sd'$ denotes the signed distance function induced by the anisotropy $\psi(x',\cdot)$.}  Let $\gamma$ be a smooth curve, with $\gamma(0)=x, \gamma(1)\in\bd E_{t-h}\h $ to be used in the definition of the geodesic distance $\sd^\psi_{E\h_{t-h}}$. Firstly, we remark that one could assume  
    \begin{equation}\label{bound p}
        \gamma([0,1])\subseteq B(x,2 c_\psi^2 |f_t\h(y)-f_{t-h}\h(y)|  )
    \end{equation}
    Indeed, if it were not the case, the lower bounds  contained in \eqref{inclusione psi palle} and \eqref{computation sd} allow us to estimate
    \begin{align} 
        \int_0^1 \psi^\circ(\gamma,\dot \gamma)\ud t\ge\dfrac 1{c_\psi}\int_0^1 |\dot{\gamma}|\ud t\ge 2c_\psi |f_t\h(y)-f_{t-h}\h(y)|
        \ge 2c_\psi\,\sd_{E_{t-h}\h}(x)\ge 2\, \sd^\psi_{E_{t-h}\h}(x), \label{estimate curve}
    \end{align}
    a contradiction for $h$ small. We can reason analogously  for $\sd'_{E\h_{t-h}}$. In particular, we can consider just curves having length $\int_0^1|\dot \gamma|\le c|f_t\h(y)-f_{t-h}\h(y)|$. Therefore, we obtain (by homogeneity)
    \begin{align*}
        \sd^\psi_{E_{t-h}\h}(x)&\le\int_0^1 \psi^\circ(\gamma,\dot \gamma)\ud t\le 	\int_0^1 \psi^\circ(x',\dot \gamma)\ud t  +  \sup_{\nu\in  {S^{N-1}},\,t\in[0,1]} |\psi(\gamma(t),\nu)-\psi(x',\nu)|   \int_0^1|\dot \gamma|\\
        &\le 	\int_0^1 \psi^\circ(x',\dot \gamma)\ud t  + c\,\omega(h) |f_t\h(y)-f_{t-h}\h(h)| ,
    \end{align*}
    and, taking the $\inf_\gamma$, we obtain $\sd^\psi_{E_{t-h}\h}(x)\le \sd'_{E_{t-h}\h}(x)+\omega(h)|f_t\h(y)-f_{t-h}\h(y)|$. The converse inequality can be proved analogously, yielding \eqref{est dist}.

    Therefore, in what follows we will consider always the anisotropy frozen in $x'$, and use $\sd'$ instead of $\sd^\psi$. Finally,  {let $p \in\bd E_{t-h}\h$  a minimizer for the definition of $\sd'_{E\h_{t-h}}(x)$}.  In the following, with $\Pi^v_{\mathbb H} z$,  {$\Pi_{\mathbb H} z$ we denote respectively} the projection on the hyperplane $\mathbb H$ of $z$ along the direction $v$  {and the  orthogonal projection of $z$ on $\mathbb H$}.
    
    \textbf{Step 2.}  {In this step we assume that $E_{t-h}\h\cap \textbf{C} $ coincides with the halfspace $\mathbb H=p+\{ z\cdot \nu\leq 0 \}$ intersected with the same cylinder and prove claim \eqref{estimate flatness}.}\\
 {To this aim, we start noticing that by translation we may assume $p=0$ and that  
    $\sd'_{\mathbb H}(z+\xi)=\sd'_{\mathbb H}(z)$ for all $z\in \R^N$ and for all $\xi$ orthogonal to $\nu$. Hence, in fact, 
    \begin{equation}\label{sd'}
    \sd'_{\mathbb H}(z)=\sd'_{\mathbb H}((z\cdot\nu)\nu)=(z\cdot\nu)\sd'_{\mathbb H}(\nu)\,.    
    \end{equation}
     Therefore, $\sd'_{\mathbb H}$ is differentiable everywhere, with $\nabla \sd'_{\mathbb H}=\sd'_{\mathbb H}(\nu)\nu$. Recalling the eikonal equation \eqref{eikonal eq}, it must hold $\sd'_{\mathbb H}(\nu)=1/\psi(x',\nu)$ and in turn, from \eqref{sd'}, 
    and choosing $z=x$, we have
    \begin{equation}\label{id dis}
        \sd'_{\mathbb H}(x)\psi(x',\nu)=x\cdot \nu=\sd_{\mathbb H }(x).
    \end{equation}}
We remark that $\sd'_{\mathbb H}(x)=\sd'_{E_{t-h}\h}(x)$ by  \eqref{eq 4.6 MugSeiSpa}, thus we conclude \eqref{estimate flatness} by combining \eqref{id dis} with \eqref{computation sd}.  
    
    
    \textbf{Step 3.}  {We now conclude in the general case.}  {With the notation introduced at the end of Step~1, set} $\nu=\nu_{E_{t-h}\h} {(p)}$, and consider the half-space $\H=p+\{ z\cdot \nu\le 0 \}$ and $w:=x'-\Pi_{\H}(x')$  {as depicted in Figure \ref{fig 1}}. We shall prove that $$|w|\le \omega(h)|f_t\h(y)-f_{t-h}\h(y)|.$$ 
    To see this, we start by remarking that \eqref{eq 4.8 MugSeiSpa} implies 
    \[ |e_N-e_N(e_N\cdot \nu_{E_{t-h}\h} )|\le \omega(h)\qquad \text{in }\bd E_{t-h}\h\cap \textbf{C},\]
    implying $e_N\cdot \nu_{E_{t-h}\h} \ge 1-\omega(h)$, and thus, for any versor $v$ tangent to $\bd E_{t-h}\h\cap \textbf{C}$ one has $|v\cdot e_N|\le \omega(h)$. Therefore, we have $(x'-p)\cdot e_N\le \omega(h)|x'-p|$ and also 
    \begin{align*}
        \dfrac{x'-p}{|x'-p|}\cdot \nu &= \dfrac{x'-p}{|x'-p|}\cdot \left( e_N (\nu\cdot e_N ) + \nu-e_N (\nu\cdot e_N) \right)\\
        &\le \omega(h) + | \nu-e_N (\nu\cdot e_N)|=\omega(h)+ \left( 1-| \nu\cdot e_N|^2 \right)^{1/2}\\
        &\le 3\sqrt{\omega(h)},
    \end{align*}
    
    
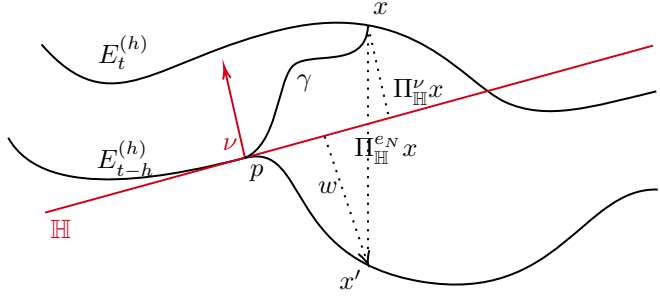
\begin{figure}[h]
\tikzset{every picture/.style={line width=0.75pt}} 
\begin{tikzpicture}[x=0.75pt,y=0.75pt,yscale=-0.6,xscale=0.6]

\draw    (72,50.5) .. controls (132,119.5) and (168.8,62.02) .. (276.8,37.41) .. controls (384.8,12.8) and (426.8,83) .. (461.8,99) .. controls (496.8,115) and (534.8,103) .. (588.8,90) ;
\draw    (44,129) .. controls (76,186) and (189,159.5) .. (243,146) .. controls (297,132.5) and (275,230.6) .. (391,250) .. controls (507,269.4) and (527,170.6) .. (590,172.6) ;
\draw  [dash pattern={on 0.84pt off 2.51pt}]  (346.4,34.71) -- (346.4,235.5) ;
\draw [color={rgb, 255:red, 208; green, 2; blue, 27 }  ,draw opacity=1 ]   (75,192) -- (586,51.5) ;
\draw  [dash pattern={on 0.84pt off 2.51pt}]  (346.4,34.71) -- (364,112) ;
\draw [color={rgb, 255:red, 208; green, 2; blue, 27 }  ,draw opacity=1 ]   (225.84,70.66) -- (243,146) ;
\draw [shift={(225.4,68.71)}, rotate = 77.17] [color={rgb, 255:red, 208; green, 2; blue, 27 }  ,draw opacity=1 ][line width=0.75]    (10.93,-3.29) .. controls (6.95,-1.4) and (3.31,-0.3) .. (0,0) .. controls (3.31,0.3) and (6.95,1.4) .. (10.93,3.29)   ;
\draw  [dash pattern={on 0.84pt off 2.51pt}]  (310,128) -- (345.76,233.61) ;
\draw [shift={(346.4,235.5)}, rotate = 251.29] [color={rgb, 255:red, 0; green, 0; blue, 0 }  ][line width=0.75]    (10.93,-4.9) .. controls (6.95,-2.3) and (3.31,-0.67) .. (0,0) .. controls (3.31,0.67) and (6.95,2.3) .. (10.93,4.9)   ;
\draw    (243,146) .. controls (270,134) and (270,86) .. (282,70) .. controls (294,54) and (345,74) .. (346.4,34.71) ;

\draw (77,195.4) node [anchor=north west][inner sep=0.75pt]  [color={rgb, 255:red, 208; green, 2; blue, 27 }  ,opacity=1 ]  {$\mathbb{H}$};
\draw (348.4,12.45) node [anchor=north west][inner sep=0.75pt]    {$x$};
\draw (320,237.4) node [anchor=north west][inner sep=0.75pt]    {$x'$};
\draw (116,128.4) node [anchor=north west][inner sep=0.75pt]    {$E_{t-h}^{( h)}$};
\draw (117,40.4) node [anchor=north west][inner sep=0.75pt]    {$E_{t}^{( h)}$};
\draw (245,149.4) node [anchor=north west][inner sep=0.75pt]    {$p$};
\draw (365,77.4) node [anchor=north west][inner sep=0.75pt]    {$\Pi _{\mathbb{H}}^{\nu } x$};
\draw (302,163.4) node [anchor=north west][inner sep=0.75pt]    {$w$};
\draw (332.5,125.15) node [anchor=north west][inner sep=0.75pt]    {$\Pi _{\mathbb{H}}^{e_{N}} x$};
\draw (284,73.4) node [anchor=north west][inner sep=0.75pt]    {$\gamma $};
\draw (222,127.4) node [anchor=north west][inner sep=0.75pt]  [color={rgb, 255:red, 208; green, 2; blue, 27 }  ,opacity=1 ]  {$\nu $};

\end{tikzpicture}
\caption{The situation in the proof of the lemma.}
\label{fig 1}
\end{figure}
    by choosing $h$ small. Up to defining $\sqrt{\omega}$ as $\omega$, using the previous estimate and the bounds \eqref{bound p} we  see that 
    \begin{equation}\label{bound planes}
        |w|=|x'-p|  \left(\dfrac{x'-p}{|x'-p|}\cdot \nu\right)\le \omega(h)|x'-p|\le \omega(h)|f_t\h(y)-f_{t-h}\h(y)|.
    \end{equation}
    We now remark that $\sd'_{E_{t-h}\h}(x)= \sd'_{\H}(x)$ (by convexity of the anisotropy $\psi(x',\cdot)$) and so, applying the previous step to $\H$  {and using also \eqref{est dist},} we get
    \begin{equation*}
        \left\lvert  \sd^\psi_{E_{t-h}^{(h)}}(x)\, \psi(x,\nu_{E_{t}^{(h)}}(x)) \sqrt{1+|\nabla f_t\h(y)|} - |x-\Pi^{e_N}_\H x| \right\rvert \le \omega(h)|x-\Pi_\H^{e_N}x|.
    \end{equation*}
    We conclude \eqref{estimate flatness} by estimating 
    \[ \big\lvert |x-\Pi_\H^{e_N}x|-|x-x'|\Big\rvert \le |x'-\Pi_\H^{e_N}x|=|w|/ |\nu\cdot e_N|\le \dfrac{\omega(h)}{1-\omega(h)}|f_t\h(y)-f_{t-h}\h(y)|,\]
    where we used \eqref{bound planes}.  We conclude the proof by a simple change of coordinates and using \eqref{estimate flatness} to find
    \begin{align*}
        \bigg\lvert  &\int_{\bd E_t\h\cap \textbf{C}}  \psi(x,\nu_{E_t\h}(x)) \, \sd^\psi_{E_{t-h}\h}(x)\udH -\int_{B_{h^\beta/2}}f_t\h(y)-f_{t-h}\h(y)\ud y \bigg\rvert\\
        &=\bigg\lvert \int_{B_{ {h^\beta}/2}} \psi((y,f_t\h(y)),\nu_{E_t\h}(y,f_t\h(y))) \, \sd^\psi_{E_{t-h}\h}(y,f_t\h(y))\sqrt{1+|\nabla f_t\h(y)|^2} -(f_t\h(y)-f_{t-h}\h(y))\ud y \bigg\lvert \\
        &\le \omega(h)\int_{B_{h^{\beta}/2}} |f_t\h-f_{t-h}\h|\ud y.
    \end{align*}
\end{proof}

Finally, we are able to prove that the error generated by approximating the discrete velocity with $v_h$ goes to zero as $h\to 0$.  {We follow the lines of \cite[Proposition 2.2]{LucStu}.}
\begin{prop}[Error estimate]\label{proposition 2.2 LucStu}
    Under the hypothesis of Lemma \ref{lemma computations}, the error in the discrete curvature equation vanishes in the limit $h\to 0$, namely
    \begin{equation}\label{error estimate on the velocity}
        \lim_{h\to 0}\,\left\lvert\frac 1h\int_0^T \left( \int_{E_t \h} \eta\ud x-\int_{E_{t-h} \h} \eta\ud x \right)\ud t-\int_0^T\int_{\bd E_{t}^{(h)}}	\psi(x,\nu_{E_{t}^{(h)}})v_h\eta\udH(x)\ud t	 \right\rvert = 0 
    \end{equation}
    for all $\eta\in C^1_c( \R^N\times [0,T))$.
\end{prop}
\begin{proof}
    We fix $t\in[2h,\infty)$ and $\alpha\in(\frac 12,\frac{N+2}{2N+2})$. For any point $x_h\in\bd E_t\h$ we define the open set $A_{x_h}$ defined as follows:
    \begin{itemize}
        \item[(i)] if \eqref{eq 4.6 MugSeiSpa} holds, we set $A_{x_h}=\textbf{C}_{h^\beta/2}(x_h,\nu)$, with the notations of Corollary \ref{estimates on flat sets};
        \item[(ii)] otherwise we set  $A_{x_h}=B(x_h,c_\infty\sqrt h)$, where $c_\infty$ is the constant of Lemma \ref{L infty estimate}.
    \end{itemize}
    By Lemma \ref{L infty estimate}, the family $\{ A_{x_h} : x_h\in\bd E_t\h \}$ is a covering of $E_t\h\triangle E_{t-h}\h.$ By a simple application of Besicovitch's theorem (see e.g. \cite{Mag-book}), we find a finite collection of points $I\subseteq \bd E_t\h$ such that $\{ A_{x_h} \}_{x_h\in I}$ is a covering of  $E_t\h\triangle E_{t-h}\h$ with the finite intersection property. We proceed to estimate \eqref{error estimate on the velocity} on each $A_{x_h}$ belonging to this family.\\[1ex]
    \textit{Estimate in case (i)} We use Proposition \ref{estimates on flat sets} to deduce
    \begin{align}
        \bigg\lvert &\int_{A_{x_h}} (\chi_{E_t\h}-\chi_{E_{t-h}\h}) \eta\ud x -\int_{\bd E_t\h\cap A_{x_h}}	\psi(x,\nu_{E_{t}^{(h)}}) \sd^\psi_{E_{t-h}\h}\eta\udH \bigg\rvert \nonumber\\
        &\le |\eta(x_h,t)|\bigg\lvert \int_{A_{x_h}} (\chi_{E_t\h}-\chi_{E_{t-h}\h}) -\int_{\bd E_t\h\cap A_{x_h}}	\psi(x,\nu_{E_{t}^{(h)}}) \sd^\psi_{E_{t-h}\h}\udH \bigg\rvert \nonumber\\
        &+\ \bigg\lvert \int_{A_{x_h}} (\chi_{E_t\h}-\chi_{E_{t-h}\h})(\eta-\eta(x_h,t)) -\int_{\bd E_t\h\cap A_{x_h}}	(\eta-\eta(x_h,t)) \, \psi(x,\nu_{E_{t}^{(h)}})\, \sd^\psi_{E_{t-h}\h} \udH \bigg\rvert\nonumber\\
        &\le C(\omega(h)\|\eta\|_\infty+h^\beta \|\nabla \eta\|_\infty) \int_{A_{x_h}}|\chi_{E_t\h}-\chi_{E_{t-h}\h}|\udH+c h^\beta\|\nabla \eta\|_\infty P(E_t\h, A_{x_h}).\label{eq 4.13 MugSeiSpa}
    \end{align}
    \textit{Estimate in case (ii)} By assumption $\exists y\in B_{c_\infty\sqrt h}(x_h)\cap (E_t\h\triangle E_{t-h}\h)$  such that $|v_h(t,y)|>h^{\alpha-1}$. We can assume \textit{wlog} $y\in E_t\h$. We then have $B(y,h^\alpha/(2c_\psi))\subseteq \R^N\setminus E_{t-h}\h$ and $\sd^\psi_{E_{t-h}\h}>h^\alpha/(2c_\psi^2)$ on $B(y,h^\alpha/(2c_\psi))$. Since $h^\alpha<<h^{1/2}$, we can use the density estimates of Lemma \ref{lemma density estimate} to deduce
    \begin{equation*}
        c h^{(N+1)\alpha-1}\le \int_{B(y,h^\alpha/(2c_\psi))\cap (E_t\h\triangle E_{t-h}\h)} |v_h|\ud x.
    \end{equation*}
    Analogously, recalling also Lemma \ref{L infty estimate}, we deduce
    \begin{equation*}
        \int_{B(x_h,c_\infty\sqrt h)\cap\bd E_t\h} |\psi(x,\nu_{E_{t-h}\h})\,\sd^\psi_{E_{t-h}\h} |\udH(x)\le ch^{\frac N2}.
    \end{equation*}
    Combining the two previous equations and $B(y,h^\alpha/(2c_\psi))\subseteq B(y,c\sqrt h)$, we infer
    \begin{align}
        \int_{A_{x_h}}|\chi_{E_t\h}&-\chi_{E_{t-h}\h}|+\int_{A_{x_h}\cap \bd E_t\h} |\psi(x,\nu_{E_{t-h}\h})\,\sd^\psi_{E_{t-h}\h}| \udH \nonumber\\
        & \le  ch^{\frac N2 - (N+1)\alpha+1}\int_{A_{x_h}\cap (E_t\h\triangle E_{t-h}\h)} |\psi(x,\nu_{E_{t-h}\h})v_h|.\label{eq 4.16 MugSeiSpa}
    \end{align}
    Summing over $x_h\in I$ both  \eqref{eq 4.13 MugSeiSpa} and \eqref{eq 4.16 MugSeiSpa}, and using the local finiteness of the covering, we get
    \begin{align*}
        \bigg\lvert &\int (\chi_{E_t\h}-\chi_{E_{t-h}\h}) \eta \ud x -\int_{\bd E_t\h}	\psi(x,\nu_{E_{t}^{(h)}}) \sd^\psi_{E_{t-h}\h}\eta \udH \bigg\rvert \nonumber\\
        &\le \sum_{x_h\in I}\bigg\lvert \int_{A_{x_h}} (\chi_{E_t\h}-\chi_{E_{t-h}\h}) \eta \ud x -\int_{\bd E_t\h\cap A_{x_h}}	\psi(x,\nu_{E_{t}^{(h)}}) \sd^\psi_{E_{t-h}\h}\eta \udH \bigg\rvert \nonumber\\
        &\le c\left(\omega(h)\|\eta \|_\infty+h^\beta\|\nabla \eta \|_\infty + h^{\frac N2-(n+1)\alpha+1}\|\eta 
        \|_\infty    \right) \cdot \nonumber\\
        &\quad \cdot\left( P(E_t\h)+|E_t\h\triangle E_{t-h}\h|+\int_{E_t\h\triangle E_{t-h}\h}|v_h| \right)
    \end{align*}
    where the last constant $c$ depends on $N,\psi$. We then use Lemma \ref{lemma 1.1 LucStu}, \eqref{eq 3.18 MugSeiSpa} and \eqref{eq 3.19 MugSeiSpa} to conclude
    \begin{align*}
        \bigg\lvert \int_{2h}^T & \frac 1h\left( \int_{E_t \h} \eta \ud x-\int_{E_{t-h} \h} \eta \ud x \right) -\int_h^T   \int_{\bd E_t\h}	\psi(x,\nu_{E_{t}^{(h)}}) v_h \eta \udH \bigg\rvert  \\
        & \le c\left(\omega(h)\|\eta \|_\infty+h^\beta\|\nabla \eta \|_\infty + h^{\frac N2-(n+1)\alpha+1}\|\eta 
        \|_\infty    \right),
    \end{align*}
    where $c=c(E_0,f,T,\psi)$ and $T$ is chosen such that spt$\,\eta  {\subset\joinrel\subset} \R^N\times [0,T]$. The conclusion follows using the definition of $\alpha$ and taking the limit $h\to 0$.
    
\end{proof}

The proof of our main theorem of this section is now a consequence of the previous results. In particular, hypothesis \eqref{convergenza senso misure} and \eqref{convergenza perimeteri} imply that the discrete flow converges to the flat flow in the sense of varifolds and this allows to prove \eqref{legge curvatura}, while \eqref{legge velocita} is a consequence of Proposition \ref{proposition 2.2 LucStu}. In order to prove the convergence of the approximations in time of the forcing term, we need to require additionally that \eqref{hp f 1} holds.

\begin{proof}[Proof of Theorem \ref{theorem existence distributional solutions}]
Firstly, combining \cite[Theorem 4.4.2]{Hut} 
with the bounds contained in \eqref{L2 bound curvatures} and in Proposition \ref{L2 bound velocity}, we conclude the existence of functions $v,H^\phi,\tilde f:\R^N\times [0,\infty) \to \R$  satisfying
    \[ \int_0^T\int_{\bd E_t} |v|^2+|H^\phi|^2+|\tilde f|^2\udH\ud t \le C_T\]
and the following properties
\begin{align}
    \lim_k \int_0^T\int_{\bd E_t^{(h_k)}} v_{h_k} \eta \udH\ud t &=\int_0^T\int_{\bd E_t}  \eta v \udH\ud t\nonumber \\
    \lim_k \int_0^T\int_{\bd E_t^{(h_k)}} F_{h_k}(x,t) \eta \udH\ud t &=\int_0^T\int_{\bd E_t}  \eta \tilde f \udH\ud t\nonumber\\
    \lim_k \int_0^T\int_{\bd E_t^{(h_k)}}  H^\phi_{E_t^{(h_k)}}\eta \udH\ud t&=\int_0^T\int_{\bd E_t}   \eta H^\phi \udH\ud t\label{eq 4.4 MugSeiSpa},
\end{align}
for any $ \eta\in C^0_c(\R^N\times[0,T))$. We now employ an approximation procedure to prove that $H^\phi(\cdot,t)$ is the $\phi-$mean curvature of $E_t$ for a.e. $t\in [0,\infty)$, following the lines of \cite{LucStu, MugSeiSpa}.  Fixed $t\in [0,+\infty)$ and $\varepsilon>0$, set $\nu_\varepsilon$ a continuous function such that $\int_{\bd E_t} (\nu_{E_t}-\nu_\varepsilon)^2\udH <\varepsilon.$ Then, by \eqref{convergenza senso misure} one could  prove that  $\lim_{k\to\infty} \int_{\bd E_t^{(h_k)}} (\nu_{E_t^{(h_k)}}-\nu_\varepsilon)^2\udH<\varepsilon.$ Considering test functions in \eqref{eq 4.4 MugSeiSpa} of the form $\eta(x,t)=a(t) g(x)$, one has for a.e. $t\in [0,+\infty)$ 
    \[\lim_k \int_{\bd E_t^{(h_k)}}  H^\phi_{E_t^{(h_k)}}g  \udH=\int_{\bd E_t}   H^\phi g \udH.\]
Thus, for a.e. $t\in [0,+\infty)$  and for any $X\in C^0_c(\R^N; \R^N)$ it holds
    \[\lim_k \int_{\bd E_t^{(h_k)}}  H^\phi_{E_t^{(h_k)}}\nu_{E_t^{(h_k)}}\cdot X \udH=\int_{\bd E_t}   H^\phi\nu_{E_t}\cdot X \udH\]
by approximating the normal vectors of ${E_t^{(h_k)}}$ with $\nu_\varepsilon$. Furthermore, by the convergence \eqref{convergenza senso misure} and the hypothesis \eqref{convergenza perimeteri} we can use the Reshetnyak's continuity theorem (see e.g. \cite[Theorem 2.39]{AmbFusPal}), ensuring
\[ \int_{\bd E_t^{(h_k)} } L(x,\nu_{E_t^{(h_k)}})\udH \to \int_{E_t}  L(x,\nu_{E_t})\udH \]
as $k\to\infty$, for any $L\in C^0_c(\R^N\times\R^N)$. We choose $L(x,\nu)=\div_\phi X$ for some $X\in C^1_c(\R^N;\R^N)$ to obtain
\begin{align*}
    \int_{\bd E_t} \div_\phi X \udH & =\lim_k \int_{\bd E_t^{(h_k)}} \div_\phi X \udH \\
    &=\lim_k \int_{\bd E_t^{(h_k)}}  X\cdot \nu_{E_t^{(h_k)}} H^\phi_{E_t^{(h_k)}}\udH\\
    &=\int_{\bd E_t}  X\cdot \nu_{E_t} H^\phi \udH,
\end{align*}  
which shows that $H^\phi(\cdot,t)$ is the $\phi-$mean curvature of the set $E_t$ for a.e. $t\in[0,+\infty).$ Moreover, we remark that $F_{h_k}(x,t)\to f(x,t)$ for every $(x,t)$, thus for any test function $\eta\in C^0_c(\R^N\times [0,+\infty))$ and $t\in[0,+\infty)$ we have

\[ \begin{split}
    \bigg\lvert \int_{\bd E_t\h} F_{h_k}(x,t) \eta(x,t) \udH_x&-\int_{\bd E}f\eta\udH_x \bigg\rvert \le \left\lvert \int_{\bd E_t\h} F_{h_k} \eta -\int_{\bd E_t} F_{h_k} \eta \right\rvert + \int_{\bd E_t} |F_{h_k}-f| \eta \\
    &\le \|f\|_\infty\|\eta\|_\infty \left( P(E_t\h)-P(E_t) \right)+\int_{\bd E_t} |F_{h_k}-f| \eta \to 0
\end{split}
\]
applying the dominated convergence theorem and recalling Lemma \ref{lemma a priori estimate}. Thus, $\tilde f=f$.  We then prove \eqref{legge curvatura} by passing to the limit in the Euler-Lagrange equation \eqref{EL equation}.

To prove  \eqref{legge velocita} 	we employ Proposition \ref{proposition 2.2 LucStu}: for every $\eta\in  C^0_c(\R^N\times[0,T))$, by a change of variables we have that
\begin{align*}
    \int_h^T &\left[  \int_{E_t^{ (h) }   } \eta  \ud x-  \int_{ E_{t-h}^{(h)} } \eta  \ud x \right] \ud t=\int_h^T \int_{E_t^{(h)}} \left( \eta (x,t)-\eta (x,t-h) \right)\ud x\ud t-h \int_{E_0} \eta \ud x
\end{align*}
where we have used that $E^{(h)}_t=E_0$ for $t\in [0,h)$.  Therefore, a simple convergence argument yields
\[ \lim_{h\to 0}  \frac 1h \int_h^T \left[  \int_{E_t^{(h)}} \eta  \ud x-  \int_{E_{t-h}^{(h)}}\eta  \ud x \right]\ud t   =-\int_h^T \bd_t \eta (x,t)\ud x\ud t - \int_{E_0} \eta. \]
Combining the previous estimate with Proposition \ref{proposition 2.2 LucStu} and passing to the limit, we obtain \eqref{legge velocita}.	
\end{proof}


\section{Viscosity solutions}
In this section we will prove the existence of another weak notion of solution for the mean curvature flow starting from a compact set.  We will follow the so-called level set approach based on the theory of viscosity solution. We recall that in the first part we work with the standing assumptions of the paper \eqref{standing hp}.  Additionally, we require \eqref{hp lip psi 1}.

\subsection{The discrete scheme for unbounded sets}\label{sect evolution unbounded sets}
In this short subsection we will define the discrete evolution scheme for unbounded sets having compact boundary. The idea would be to define this evolution simply as the complement of the evolution of the complementary set, but since the anisotropies we are considering are not symmetric, we need additional care. 

We recall that, given an anisotropy $\phi,$ we define $\tilde \phi(x,\nu):=\phi(x,-\nu)$. This anisotropy has all the properties of the original one, concerning regularity and bounds. We start remarking the following simple fact. One can see that $\dist^\psi(x,y)=\dist^{\tilde \psi}(y,x), $ since for any curve $\gamma\in W^{1,1}([0,1];\R^N), \gamma(0)=x, \gamma(1)=y,$ a simple change of variable yields
\[ \int_0^1 \psi^\circ(\gamma(t),\dot\gamma(t))\ud t = \int_0^1 \psi^\circ\left(\gamma(1-t),- \frac{\ud}{\ud t}\left( \gamma(1-t)\right)\right)\ud t = \int_0^1 \widetilde{(\psi^\circ)}(\eta(t),\dot\eta(t))\ud t, \]
for $\eta(t)=\gamma(1-t)$, once one sees that $$ \widetilde{(\psi^\circ)}(\cdot,\nu)=\sup_{\psi(\cdot,\xi)\le 1}\xi\cdot (-\nu)=\sup_{\tilde\psi(\cdot,-\xi)\le 1}(-\xi)\cdot \nu=(\tilde\psi)^\circ(\cdot,\nu).$$ Therefore, by definition of signed distance we have
\begin{equation}\label{invers dist}
    \sd^\psi_E(x)=-\sd^{\tilde\psi}_{E^c}(x).
\end{equation}
For every compact set $F$ and $h>0,t\ge 0$, we will denote by $\tilde T_{h,t}^\pm F$ the maximal and the minimal solution to problem \eqref{problema discreto}, according to Lemma \ref{existence discrete pb}  with $\P$ and $\sd^\psi$, respectively, replaced by $P_{\tilde \phi}$ and $\sd^{\tilde \psi}$. Finally, for every set $E$ with compact boundary we define
\begin{equation}\label{pb discr compl}
    T_{h,t}^\pm E:=\left(  \tilde T^\mp_{h,t} E^c \right)^c.
\end{equation}
As in the case for compact sets, we set $T_{h,t} E:=T_{h,t}^- E.$ Given an open, unbounded set $E_0$ having compact boundary, we can then define the discrete flow $\{E_t\h\}_{t\ge 0}$ as follows: $E\h_t:=E_0$ for $t\in [0,h)$ and
\[ E\h_t=T_{h,t} E\h_{t-h},\quad \forall t\in[h,+\infty). \]
 One easily checks that analogous results to Lemmas \ref{comparison principle}, \ref{lemma a priori estimate} and \ref{lemma estimates on balls} hold also for this problem. We state the corresponding results.

 \begin{lemma}\label{comparison principle, unbounded}
     Let $F_1\subseteq F_2$ be open, unbounded sets with compact boundary and fix $h>0,t\ge0$. Then, $T_{h,t}F_1\subseteq T_{h,t}F_2$.
 \end{lemma}

\begin{lemma}
    For any $T>0$ there exists a constant $ C_T(\phi,\psi,f,T)$ such that for every $R>0$ the following holds. If the initial open set $E\supset B_R^c$, then $E_t\h\supset B_{C_T R}^c$ for all $t\in[0,T].$
\end{lemma}

\begin{lemma}
		For every $R_0>0$ there exist $h_0(R_0)>0$ and $C(R_0,\phi,\psi,f)>0$ with the following property: For all $R\ge R_0$, $h\in(0,h_0)$, $t>0$ and $x\in\R^N$ one has
		\begin{equation*}
			 T_{h,t}(\left(B_R(x)\right)^c)\subseteq \left( B_{R-Ch}(x)\right)^c.
		\end{equation*}
\end{lemma}

We now state a comparison principle between bounded and unbounded sets, following the line of \cite[Lemma 6.10]{ChaMorPon15}.
\begin{lemma}\label{comparison principle, bounded-unbounded}
    Let $E_1$ be a compact set and let $E_2$ be an open, unbounded set, with compact boundary, and such that $E_1\subseteq E_2$. Then, for every $h\in(0,1), t\ge 0$ it holds $T^\pm_{h,t} E_1\subseteq T^\pm_{h,t}E_2.$
\end{lemma}
\begin{proof}
    We fix $h\in(0,1),t\in[0,T]$ for $T>0$. Set $R>0$ such that $E_1,E_2^c\subseteq B_R$ and note that by Lemmas \ref{comparison principle} and \ref{lemma a priori estimate} (applied to $P_{\tilde\phi}$ instead of $\P$) we get 
    \begin{equation}\label{eq 6.9 Nonlocal}
        \left( T^+_{h,t}E_2 \right)^c\subseteq \tilde T^-_{h,t}E_2^c\subseteq T^-_{h,t} B_R\subseteq B_{C_TR},
    \end{equation}
    for some $C_T(\phi,\psi,f,T)$. Since $\tilde T^-_{h,t}E_2^c$ is the minimal solution of
    \[\min\left\lbrace P_{\tilde\phi}(E) + \dfrac 1h \int_{E} \sd^{\tilde \psi}_{E^c_2}(x)\ud x - \int_{E} F_h(x,t)\ud x \right\rbrace,\]
    considering the change of variables $\tilde E=E^c$ and using \eqref{invers dist}, we easily conclude that $T^+_{h,t}E_2=\left( \tilde T^-_{h,t}E_2^c \right)^c$ is the maximal solution of
    \[\min\left\lbrace P_{\phi}(\tilde E) + \dfrac 1h \int_{B_{C_TR}} \sd^{\psi}_{E_2} -\dfrac 1h \int_{\tilde E^c} \sd^\psi_{E_2}- \int_{\tilde E^c} F_h(x,t)\ud x \right\rbrace- \dfrac 1h \int_{B_{C_TR}} \sd^{\psi}_{E_2}.\]
    we then note that 
    \[  \int_{B_{C_TR}} \sd^{\psi}_{E_2}= \int_{\tilde E} \sd^{\psi}_{E_2}\chi_{B_{C_TR}}+ \int_{\tilde E^c} \sd^{\psi}_{E_2}, \]
    for every $\tilde E$ such that $\tilde E^c\subseteq B_{C_TR}.$ By \eqref{eq 6.9 Nonlocal}, we conclude that $T_{h,t}^+E_2$ is the maximal solution of
     \begin{equation}\label{eq 6.10 Nonlocal}
         \min\left\lbrace P_{\phi}(\tilde E) + \dfrac 1h \int_{\tilde E} \sd^\psi_{E_2}\chi_{B_{C_TR}}- \int_{\tilde E^c} F_h(x,t)\ud x\ :\ \tilde E^c\subseteq B_{C_TR} \right\rbrace.
     \end{equation}
    Analogously, one proves that $T_{h,t}^-E_2$ is the minimal solution of \eqref{eq 6.10 Nonlocal}. Finally, we remark that $\sd^\psi_{E_s}\chi_{B_{C_TR}}\le \sd^\psi_{E_1}$ and that $T_{h,t}^\pm E_1\cup T_{h,t}^\pm E_2, T_{h,t}^\pm E_1\cap T_{h,t}^\pm E_2 $ are both admissible competitors for \eqref{eq 6.10 Nonlocal}, one argues exactly as in the proof of Lemma \ref{comparison principle} to conclude $T_{h,t}^\pm E_1\subseteq T_{h,t}^\pm E_2.$
\end{proof}

\subsection{The level set approach}
 {We recall that in this section we assume  \eqref{standing hp}, \eqref{hp lip psi 1}. }Consider a function $u:\R^{N}\times [0,+\infty)\to \R$ whose  {spatial} superlevel sets  {$\{u(\cdot,t)\ge s\}$}  evolve according to the mean curvature equation 
\[ V(x,t)=-\psi(x,\nu_{\{u(\cdot,t)\ge s \}})\left( H^\phi_{\{u(\cdot,t)\ge s  \}}(x)-f(x,t)\right)\quad {\text{for }x\in\bd \{ u(\cdot,t)\ge s\}}. \]
The function $u$ then satisfies (recalling that $-\nabla u/|\nabla u|$ is the \textit{outer normal vector} to the superlevel set $\{u(\cdot,t)\ge u(x,t)\}$) the equation
\begin{align*}
	\bd_t u=|\nabla u|V(x)&=-\psi(x,-\nabla u) \left( H^\phi_{\{u(\cdot,t)\ge u(x,t)\}}(x)-f(x,t)\right)\nonumber\\
	&=-\psi(x,-\nabla u) \left(\div\nabla_p \phi(x,-\nabla u) -f(x,t)\right)\nonumber\\
	&=-\psi(x,-\nabla u) \left( \sum_i \bd_{x_i}\bd_p\phi(x,-\nabla u)-\nabla_p^2\phi(x,-\nabla u)\,:\, \nabla^2 u  -f(x,t)\right)\nonumber\\
	:&=-\psi(x,-\nabla u)\left(  H(x,\nabla u,\nabla^2 u)-f(x,t) \right),
\end{align*}
where we defined the Hamiltonian  $H\ :\ \R^N\times \R^N\setminus\{0\}\times  {Sym_N}\to \R$ as 
\begin{equation}
	\label{def hamiltonian}
	H(x,p,X):= \sum_i \bd_{x_i}\bd_{p_i}\phi(x,-p)-\nabla_p^2\phi(x,-p)\,:\, X.
\end{equation}
We therefore focus on solving the parabolic Cauchy problem 
\begin{equation}\label{Hamilton-Jacobi eq}
	\begin{cases}
		\bd_t u+\psi(x,-\nabla u)\left(  H(x,\nabla u,\nabla^2 u)-f(x,t) \right)=0\\
		u(\cdot,t)=u_0.
	\end{cases}
\end{equation}
The appropriate setting for  this type of geometric evolution equations is the one of viscosity solutions, in the framework of \cite{GigGotIshSat, IshSou} (see also \cite{ChaMorPon15}).  We will focus on the evolution of sets with compact boundary on compact time intervals of the form $[0,T]$. We now define the notion of admissible test function.  {In the following, with a small abuse of language, we will say that a function $u:\R^n\times [0,T]\to \R$  is constant outside a compact set if there exists a compact set $K\subset \R^N$ such that $u(\cdot, t)$ is constant in $\R^N\setminus K$ for every 
$t\in [0,T]$ (with the constant possibly
depending on $t$).}

\begin{defin}\label{def visco sol}
	Let $\hat z=(\hat x,\hat t)\in\R^N\times (0,T)$ and let $A\subseteq (0,T)$ be any open interval containing $\hat t$. We will say that $\eta \in C^0(\R^N\times \bar A)$ is admissible at the point $\hat z$ if it is of class $C^2$ in a neighborhood of $\hat z$, if it is constant out of a compact set, and,  in case $\nabla \eta (\hat z)=0,$ the following holds: for all $(x,t)\in\R^N\times A$, and there exist numbers $a,b>0$ such that
    \[  {|\eta (x,t)-\eta (\hat z)-\eta _t(\hat z)(t-\hat t)| \le a |x-\hat x|^3+ b|t-\hat t|^2. }\] 
\end{defin}

 {We then recall one of the equivalent definitions of viscosity solutions.}

\begin{defin}
	An upper semicontinuous function $u : \R^N\times [0,T]\to \R$ (in short, $u\in usc(\R^N\times [0,T])$), constant outside a compact set, is a viscosity subsolution of the Cauchy problem \eqref{Hamilton-Jacobi eq} if $u(\cdot,0)\le u_0$ and for all $z:=(x,t)\in\R^N\times  {(0,T)}$ and all $C^\infty-$test functions $\eta $ such that $\eta $ is admissible at $z$ and $u-\eta $ has a maximum at $z$ (in the domain of definition of $\eta $) the following holds:
	\begin{itemize}
		\item[i)] If $\nabla \eta (z)=0$, then it holds
		\begin{equation}\label{eq viscosa degen}
			\eta _t(z)\le 0
		\end{equation}
		\item[ii)] If $\nabla \eta (z)\neq 0$, then 
		\begin{equation}\label{eq viscosa}
			\bd_t \eta (z)+\psi(z,-\nabla \eta (z)) \left( H(z,\nabla \eta (z),\nabla^2 \eta (z))-f(z,t)  \right) \le 0.
		\end{equation}  
	\end{itemize}
	A lower semicontinuous function $u : \R^N\times [0,T]\to \R$ (in short, $u\in lsc(\R^N\times [0,T])$), constant outside a compact set, is a viscosity supersolution of the Cauchy problem \eqref{Hamilton-Jacobi eq} if $u(\cdot,0)\ge u_0$ and for all $z:=(x,t)\in\R^N\times[0,T]$ and all $C^\infty-$test functions $\eta $ such that $\eta $ is admissible at $z$ and $u-\eta $ has a minimum at $z$ (in the domain of definition of $\eta $) the following holds:
	\begin{itemize}
		\item[i)] If $\nabla \eta (z)=0$, then $\eta _t(z)\ge 0$;
		\item[ii)] If $\nabla \eta \neq 0$ then 
		$$\bd_t \eta (z)+\psi(z,-\nabla \eta (z)) \left( H(z,\nabla \eta (z),\nabla^2 \eta (z))-f(z,t)  \right) \le 0. $$
	\end{itemize}
	Finally, a function $u$ is a viscosity solution for the Cauchy problem \eqref{Hamilton-Jacobi eq} if it is both a subsolution and a supersolution of \eqref{Hamilton-Jacobi eq}.
\end{defin}

\begin{oss*}
	By classical arguments, one could assume that the maximum of $u-\eta$  is strict in the definition of subsolution above (an analogous remark holds for supersolutions).
\end{oss*}

\begin{oss*}
    We remark that, if $-u$ is a subsolution to \eqref{Hamilton-Jacobi eq} with initial datum $-u_0$, then $u$ is a supersolution for \eqref{Hamilton-Jacobi eq} for the initial datum $u_0$ and where $\phi,\psi$ are replaced by $\tilde \phi,\tilde \psi$ respectively, as defined in Section \ref{sect evolution unbounded sets}.
\end{oss*}

We will first prove existence for viscosity solutions of \eqref{Hamilton-Jacobi eq} via an approximation-in-time technique, and then prove uniqueness of solutions to \eqref{Hamilton-Jacobi eq} to link the approximate solution to the mean curvature flow equation.   We would like to proceed with the classical construction of e.g. \cite{Cha,ChaMorPon15,EtoGigIsh}, but in our case the lack of continuity of the evolving functions forces us to be particularly careful with the procedure.

We use the shorthand notation of \textit{lsc} for lower semicontinuous and \textit{usc} for upper semicontinuous. Given a bounded, \textit{usc} function $v$ which is constant outside a compact set, we define the transformation 
\begin{equation}\label{def increm supersol}
    T_{h,t}^+v(x)=\sup\left\lbrace s\ :\ x\in T^+_{h,t}\{ v\ge s \} \right\rbrace.
\end{equation}
Firstly, we see that $T_{h,t}^+v(x)\in\R,$ as $v$ is bounded. Moreover, it turns out that the function $T^+_{h,t}v$ is \textit{usc}, bounded and constant outside a compact set. Indeed, definition \eqref{def increm supersol} is equivalent to
\begin{equation*}
    T^+_{h,t}v(x)=\inf\left\lbrace s\ :\ x\notin T^+_{h,t}\{ v\ge s \} \right\rbrace=\inf_{s\in\R}\left( s+\mathds{1}_{\left(T^+_{h,t}\{ v\ge s \}\right)^c}(x)\right), 
\end{equation*}
where $\mathds 1_A(x)$ is the indicatrix function of a set $A$, being 0 on the set and $+\infty$ outside. By definition,  $\mathds1_A$ is an \textit{usc} function for any open set $A$. Thus, recalling Remark \ref{open close min max sols}, in the equation above we are taking the infimum of a family of \textit{usc} functions, which is then a \textit{usc} function. The other two properties follows from the previous study of the discrete evolution. Analogously, given a bounded \textit{lsc} function $g$, we define 
\begin{equation}
    T_{h,t}^-g(x)=\sup\left\lbrace s\ :\ x\in T^-_{h,t}\{ g> s \}\right\rbrace= \sup_{s\in\R}\left( s-\mathds{1}_{T^-_{h,t}\{ g>s \}} \right),
\end{equation}
which is now  a bounded \textit{lsc} function (as $\sup$ of \textit{lsc} functions), constant outside a compact set.

We are now ready to give the definition of the discrete-in-time approximations of sub {-} and super solution to \eqref{Hamilton-Jacobi eq}. Given an initial compact set $E_0$, set $u_0$ as a (uniformly) continuous function, spatially constant outside a compact set, such that $\{u_0\ge 0\}=E_0$. We remark that for every $s\in\R$, the superlevel set $\{u_0\ge s\}$ is either compact or it is unbounded with compact boundary. Then, for $h>0$ we introduce the following family of maps as $u^\pm_h(\cdot,t)=u_0$ for $t\in[0,h)$ and
\begin{equation}\label{def subsol iter}
    u^\pm_h(\cdot,t):=T^\pm_{h,t-h}u^\pm_h(\cdot,t-h)\quad \text{for }t\ge h.
\end{equation}
We easily see that the maps above are functions (as implied by the comparison principle contained in Lemmas \ref{comparison principle}, \ref{comparison principle, unbounded} and \ref{comparison principle, bounded-unbounded}) piecewise constant in time (as $T_{h,t}^\pm=T_{h,[t/h]h}^\pm$). Moreover, by the previous remarks, we have that $u^+_h(\cdot,t)$ is an \textit{usc} function, while $u^-_h(\cdot,t)$ is a \textit{lsc} function, for every $t\in[0,+\infty)$. Some further properties of the approximating scheme are listed below.
\begin{lemma}
    For any $h>0$, $t\ge 0$ we have the following. It holds 
    \begin{equation}\label{ineq subsol}
        u^-_h(\cdot,t)\le u^+_h(\cdot,t).
    \end{equation}
    Furthermore, given any $\lambda\in\R$ and $t\ge h$ it holds 
    \begin{align}
        \{ u_h^+(\cdot,t)> \lambda \} \subseteq T^+_{h,t-h}\{ u_h^+(\cdot,t-h)&\ge \lambda  \}\subseteq \{ u_h^+(\cdot,t) \ge \lambda \}\label{inclusion level sets}\\
        \{ u_h^-(\cdot,t)> \lambda \} \subseteq T^-_{h,t-h}\{ u_h^-(\cdot,t-h)&> \lambda  \}\subseteq \{ u_h^-(\cdot,t) \ge \lambda \}\nonumber.
    \end{align}
\end{lemma}
\begin{proof}
    Fix $x\in \R^N$, $t\in[0,h)$. For any given $\sigma<u_h^-(x,h)$ we have that there exists a sequence $(s_n)\nearrow \sigma$ so that $x\in T^-_{h,t-h}\{ u_0>s_n \}\subseteq T^+_{h,t-h}\{ u_0\ge s_n \}.  $ Thus, $u_h^+(x,t)\ge \sigma.$ We then conclude by induction.  Then, \eqref{inclusion level sets} follows easily by the definition \eqref{def subsol iter}.
\end{proof}

We  then prove that the half-relaxed limits (in the spirit of \cite{BarSonSou}, see also the references therein) of the families of functions $u_h^\pm$
\begin{equation}\label{def subsol supersol}
	\begin{split}
		u^+(x,t):=& \sup_{(x_h,t_h)\to (x,t)}\limsup_{h\to 0} u_h^+(x_h,t_h)  \\
		u^-(x,t):=& \inf_{(x_h,t_h)\to (x,t)}\liminf_{h\to 0} u_h^-(x_h,t_h),
	\end{split}
\end{equation}
are (respectively) sub {-} and supersolutions in the viscosity sense of \eqref{Hamilton-Jacobi eq}, see Theorem \ref{teo sol viscosa} (note that, by definition, $u^+$ is \textit{usc}, while $u^-$ is \textit{lsc}). The proof of this result is the subject of the following section and we recall that the hypothesis required are \eqref{standing hp}, \eqref{hp lip psi 1} and $f\in C^0(\R^N\times [0,\infty))$ only.  Once the existence of sub {-} and super-solutions to the equation is settled, we need to properly define the notion of level-set solution to the mean curvature flow. To do so, we first prove  uniqueness for \eqref{Hamilton-Jacobi eq} via a comparison principle and under additional hypothesis.  Then, we show that the evolution of the zero superlevel set of the solution does not depend on the choice of the initial function $u_0$. 

We start with a comparison result between $u^+,u^-$ and $u_0$ at the initial time: it will ensure that the classical hypothesis for the comparison principle are satisfied.
We  first prove an estimate for the speed of decay of the level sets of the evolving functions. While it will only be needed in the following section, in the proof of the forthcoming Lemma \ref{lemma initial datum} we will use similar techniques, so we preferred to state it here.
\begin{lemma}
	\label{lemma estimate decay balls}
	Let $u^+(x,t)$ be the function defined in \eqref{def subsol supersol}, let $\sigma\in\R$. Assume that, for a suitable $x_0$ and  $R>0$, it holds $B(x_0,R)\subseteq \{ u^+(\cdot,t_0)\ge \sigma \}$ . Then, there exists $C=C(R,\phi,\psi,f)$ such that $B(x_0,R-C(t-t_0))\subseteq \{ u^+(\cdot,t)\ge \sigma  \}$ for every $t\le t_0+R/(2C)$. An analogous statement holds for $u^-$ by considering its open sublevel sets.
\end{lemma}

\begin{proof}	
    We focus on the case $\{ u^+(\cdot,t_0)\ge \sigma \}$ bounded, the other case being analogous.  By assumption, for any $R_0<R$, if $h$ is small enough, we have $B(x_0,R_0)\subseteq \{ u_h^+(\cdot,t_0)\ge\sigma  \}$. Set $C=C(R_0/2,\phi,\psi,f)$ as the constant of Lemma \ref{lemma estimates on balls}.  Let $R_n$ be defined recursively following law \eqref{evolution law ball}, that is $R_{n+1}=R_n-Ch$, as long as $R_n\ge R_0/2$. By simple iteration we find that $R_n=R_0-n Ch,$ 	as long as $R_n\ge R_0/2$, which can be ensured enforcing $h\,n\le {R_0}/(2C). $ Therefore, for any $t\ge t_0$ such that $t-t_0\le R_0/(2C)$, we set $n=[(t-t_0)/h]$ and send $h\to 0$ to deduce (recalling also Lemma \ref{comparison principle})
	\[ \{ u^+(\cdot,t)\ge \sigma  \} \supset B(x_0,R_0-C (t-t_0)).  \]
	Since the choice of $R_0$ is arbitrary, we conclude.
\end{proof}

We are now ready to prove a comparison result for the functions $u^\pm$ and a continuity estimate at the initial time $t=0$.

\begin{lemma}\label{lemma initial datum}
    For any $(x,t)\in \R^N\times [0,+\infty)$ it holds
    \[  u^-(x,t)\le u^+(x,t). \]
    Moreover $u^-(\cdot,0)= u^+(\cdot,0)= u_0,$ so that there exists a modulus of continuity $\omega$ such that $\forall x,y\in \R^N$
    \[ u^+(x,0)-u^-(y,0)\le \omega(|x-y|). \]
\end{lemma}
\begin{proof} 
    The proof of the first inequality essentially follows from \eqref{ineq subsol} and the definition of $u^\pm$. To prove the equality at the initial time $t=0$, we start by remarking that $u^+(\cdot,0)\ge u_0$ as can be seen  taking sequences of the form $(x_h,0)$ in \eqref{def subsol supersol}. Then, consider   $\omega$ as  a continuous, strictly increasing modulus of continuity for $u_0$. We can also see that $\forall \varepsilon>0\ \{u_0\le u_0(x)+\varepsilon  \} \supseteq B(x,\omega^{-1}(\varepsilon))$ by uniform continuity. Thus, reasoning iteratively as in Lemma \ref{lemma estimate decay balls} and using \eqref{inclusion level sets}, we obtain that there exists $h_0(\varepsilon)$ such that $\forall h\le h_0$ it holds 
    \[ \{ u_h^+(\cdot,t) \le u_h^+(x,0)+\varepsilon \}\supseteq \left( T_{h,t-h}^+\{ u_0>u_0(x)+\varepsilon \} \right)^c=T_{h,t-h}^-\{ u_0\le u_0(x)+\varepsilon \} \supseteq  B(x,\omega^{-1}(\varepsilon/2)), \]
    as long as $t\le (\omega^{-1}(\varepsilon)-\omega^{-1}(\varepsilon/2))/(2C)=:t_\varepsilon$, and where we recalled that $u_h^\pm(\cdot,0)=u_0$. Now, fix $\sigma>0,x\in\R^N$ such that $u(x,0)>\sigma$ and a sequence $(x_{h_k},t_{h_k})\to (x,0)$ such that $\lim_k u_{h_k}^+(x_{h_k},t_{h_k})>\sigma$. Then, for $k$ large enough $(x_{h_k},t_{h_k})\in B(x,\omega^{-1}(\varepsilon/2))\times [0,t_\varepsilon)$ and so we conclude 
    \[ \sigma < \lim_k u_h^+(x_{h_k},t_{h_k}) \le u_0(x,0)+\varepsilon.  \]
    Letting $\varepsilon\to 0$ we conclude $u(\cdot,0)^+\le u_0$. The proof for $u^-$ is essentially the same. The last claim  follows from the previous one, recalling that $\omega$ is a modulus of uniform continuity for $u_0$.
\end{proof}

In order to prove a comparison principle for \eqref{Hamilton-Jacobi eq}, we will need to assume \eqref{hp 2}. Under these additional hypotheses, we are able to prove uniqueness for the parabolic Cauchy problem \eqref{Hamilton-Jacobi eq}.  The proof of this result follows from \cite[Theorem 4.2]{GigGotIshSat}: we will just show in detail that the assumption of the aforementioned theorem hold in our case, following  \cite[Proposition 6.1]{BelPao} and \cite[pag. 463]{GigGotIshSat}.

\begin{proof}[Proof of Theorem \ref{existence and regularity of viscosity solutions}]
	The proof of this result essentially follows from \cite[Theorem 4.2]{GigGotIshSat}, combined with the existence result of Theorem \ref{teo sol viscosa}. Referring to the notation of \cite{GigGotIshSat}, we firstly remark that in our case $\Omega=\R^N$, thus the parabolic boundary of $U=\Omega\times [0,T]$ is simply $\bd_p U=\R^N\times \{0\}$. Therefore, the initial conditions $(A1)-(A3)$ are all verified by Lemma \ref{lemma initial datum}. We then define the continuous Hamiltonian $F:[0,T]\times\R^N\times (\R^N\setminus\{0\})\times M^{N\times N}\to \R$ as follows
	\begin{equation}\label{Hamiltonian Giga}
		F(t,x,p,X):=\psi(x,-p)\left(  -\sum_i\bd_{x_i}\bd_{ {p_i}}\phi(x,-p)+\nabla_p^2 \phi(x,-p):X +f(x,t) \right),
	\end{equation}
	and focus on the conditions $(F1), (F3)-(F5), (F6'), (F7), (F9), (F10)$ that $F$ must satisfy. The assumptions $(F1),$ $ (F3)-(F5), (F9)$ are easily checked. $(F6')$ follows from the Lipschitz regularity of $\phi$ and $\psi$, as $\forall t\in[0,T],x\in\R^N,|p|\ge\rho,|q|+|X|\le R$ one has
	\begin{align*}
		&|F(t,x,p,X)-F(t,x,q,X)|\le c_\psi |p-q| \left\lvert -\sum_i\bd_{x_i}\bd_{ {p_i}}\phi(x,-p)+\nabla^2_p\phi(x,-p):X  \right\rvert \\
		&+\psi(x,-q)\left\lvert-\sum_i\left( \bd_{x_i}\bd_{ {p_i}}\phi(x,-p)-\bd_{x_i}\bd_{ {p_i}} \phi(x,-q)\right)+\left( \nabla^2_p\phi(x,-p)-\nabla^2_p\phi(x,-q)\right):X\right\rvert\\
		&\le c_R|p-q|\left( 1+\dfrac 1{|p|} \right)+c_R |p-q|\le c_{R,\rho}|p-q|.	\end{align*}
	
	For $(F7)$, we remark that the first term in the parenthesis in \eqref{Hamiltonian Giga} is $0-$homogeneous in $p$, while the second one is $(-1)-$homogeneous in $p$ but $1-$homogeneous in $X$. Lastly, we sketch how to prove $(F10)$. Since it concerns the  $X$-terms, we focus simply on 
	\[ \nabla ^2_p \phi(x,-p): X=\text{tr} \left( \nabla^2_p\phi(x-,p)\,X^T\right). \]
    Multiplying by $\phi(x,-p) $, we rewrite $\phi(x,-p)\text{tr}\left( \nabla^2_p\phi(x-,p)\,X^T\right)=\text{tr}(A(x,-p)X^T),$ where $A= B - \left( \nabla_p \phi \otimes \nabla_p \phi\right),$ with $B$ being the uniformly elliptic operator $\frac 12\nabla^2_p \phi^2$. We can then factorize	$B=\tilde L\tilde L^T$, with $\tilde L$ being a nondegenerate, lower triangular matrix. Then, following the proof of \cite[Proposition 6.1]{BelPao} and \cite[pg. 463]{GigGotIshSat}, we obtain $(F10)$.
\end{proof}

Once uniqueness is settled, one can finally define the notion of \textit{level set} solution to the mean curvature flow as follows.

\begin{defin}
	Let $E_0$ be a compact initial set. Define a uniformly continuous, bounded function $u_0:\R^N\to\R$ such that $\{ u_0\ge 0 \}=E_0$. Then, let $u: \R^N\times [0,+\infty)\to\R$ be the unique continuous viscosity solution to \eqref{Hamilton-Jacobi eq} given by Theorem \ref{existence and regularity of viscosity solutions}. Then, the family  $E_t:=\{ u^+(\cdot,t)\ge 0 \}_{t\ge 0}$ 	will be called the level set solution to the mean curvature flow.
\end{defin}
This definition is well posed since the Hamiltonian defined in \eqref{def hamiltonian} satisfies the so-called \textit{geometricity condition}. Namely, one can easily check that for any $\lambda\neq 0, p\in\R^N\setminus0,q\in\R^N$ and any symmetric $N\times N$ matrix $X$ one has
\[ H(x,\lambda p,\lambda X+p\otimes q+ q\otimes p)=\dfrac \lambda{|\lambda|}H(x,p,X). \]
Thus, one can prove by classical arguments (see e.g. \cite[Remark 3.9]{ChaMorPon15}) the following result.
\begin{lemma}
    Let $u_0,\tilde u_0$ two initial data for \eqref{Hamilton-Jacobi eq} such that $\{ u_0\ge 0\}=\{ \tilde u_0\ge 0\}$. Then, denoting by $u,\tilde u$ the corresponding solutions to \eqref{Hamilton-Jacobi eq}, one has 
    \[ \{ u(\cdot,t)\ge 0\}=\{ \tilde u(\cdot,t)\ge 0\}\quad \text{ for all }t\in[0,T], \]
    and the same identity holds for the open superlevel sets.
\end{lemma}

\subsection{Proof of Theorem \ref{teo sol viscosa}}

In this section we will prove that the limiting functions $u^\pm$ are respectively a  viscosity sub {-} and supersolutions to \eqref{Hamilton-Jacobi eq}. We remark that we work assuming \eqref{standing hp}, \eqref{hp lip psi 1} and that $f\in C^0(\R^N\times [0,+\infty))$. We will be following the structure of the proof of \cite[Theorem 6.16]{ChaMorPon15}, but taking into account the weaker definition of  {$u^+$} holding in our case. We will be using the $O,o$ notations with respect to $h\to 0$ and focus on proving that  {$u^+$} is a subsolution. The proof for $u^-$ is analogous.

\begin{proof}[Proof of Theorem \ref{teo sol viscosa}]
Consider $u^+$ as defined in \eqref{def sottosoluzione}: we need to prove that it is a subsolution. In the following, we will denote $u:=u^+$ and $u_h:=u_h^+$. Let $\eta (x,t)$ be an admissible test function in $\bar z:=(\bar x, \bar t) {\in \R^N\times (0,T)}$ and assume that $(\bar x, \bar t)$ is a strict maximum point for $u-\eta $. Assume furthermore that $u-\eta =0$ in such  point. We need to show that either \eqref{eq viscosa degen} or \eqref{eq viscosa} holds at $\bar z$.\\
 {\textbf{Case 1.}} Let us first assume that $\nabla\eta  (\bar z)\neq 0$. By classical arguments, we can assume that $\bar z$ is a strict maximum point and that $\eta $ is smooth. By the definition of $u$,  there exists a sequence $\tilde z_k:=(\tilde x_{h_k},\tilde t_{h_k})\to \bar z$ such that $\lim_k u_{h_k}(\tilde z_k)= u(\bar z).$  {We remark that we can substitute the functions $u_{h_k}$ for $t>0$ with their \textit{usc} envelope in time, without changing the value of $u$. Indeed, the \textit{usc} envelope of $u_{h_k}$ is the function  at all discrete times $lh_k$ is given by
$$\max \{u_{h_k}(\cdot,(l-1)h_k), u_{h_k}(\cdot,lh_k) \}$$
and coincides with $u_{h_k}$ elsewhere. Since now $u_{h_k}$ is \textit{usc} in time and space,} by standard arguments (compare e.g. \cite[Lemma 6.1]{Bar-rev}), there exists a radius $\rho>0$ such that all functions $u_{h_k}-\eta$ achieve a local maximum in $B_\rho(\bar z)$ at points $z_{k}=(x_k,t_k)$.  Then, passing to a further subsequence we can ensure that $z_{k}\to w\in B_\rho(\bar z)$, and we use the definition of $u$ to obtain 
\[ (u-\eta)(w)\ge \limsup_k (u_{h_k}-\eta)(z_{k})\ge \limsup_k (u_{h_k}-\eta)(\tilde z_{k}) = (u-\eta)(\bar z).\]
Therefore, $w=\bar z$ by maximality. Thus we can assume that each  function $u_{h_k}-\eta$ achieves a local maximum in $B_\rho(\bar z)$ at a point $z_{h_k}=:(x_k,t_k)$ and that $u_{h_k}(z_{h_k})\to u(\bar z)$  as $k\to \infty.$ Finally, we can assume also that $\nabla \eta (x_k,t_k)\neq 0$ for  $k$ large enough. \\
 {\textbf{Step 1.} We start defining an appropriate set which is then used as a competitor for the minimality of the level sets of the functions $u_h$. From the previous computations}, one has  in particular that
\begin{equation}\label{eq 6.19 Nonlocal}
    u_h(x,t)\le \eta (x,t)+c_k
\end{equation}
where $c_k:= u_{h_k}( x_k, t_k)- \eta ( x_k, t_k) $, with equality if $(x,t)=( x_k,  t_k)$. Let $\sigma>0$ and set 
\begin{equation*}
    \eta _{h_k}^\sigma(x):= \eta (x, t_k)+c_k+\dfrac{\sigma}2 |x-x_k|^2.
\end{equation*}
Then, for all $x\in\R^N$, 
\[ u_{h_k}(x,t_k)\le \eta ^\sigma_{h_k}(x) \] 
with equality {if and only if} $x=x_k$. We set $l_k=u_{h_k}(x_k,t_k)=\eta ^\sigma_{h_k}(x_k)$. We fix $\varepsilon>0$, to be chosen later, and write $E_{\varepsilon,k}:=\left\lbrace u_{h_k}(\cdot,t_k-h_k)\ge l_k-\varepsilon  \right\rbrace$. We define\footnote{ We need to define the sets $W_\varepsilon$ in this way (compare the different definition in \cite{ChaMorPon15}) since firstly, we can not rule out that the inclusions in \eqref{inclusions W_e} are strict, and secondly it is not clear if otherwise $|W_\varepsilon|>0$.}
\begin{equation}
    W_\varepsilon:=\left(  T_{h,t_k-h_k}^+ E_{\varepsilon,k} \right)\setminus\left\lbrace \eta ^\sigma_{h_k}(\cdot)> l_k+\varepsilon  \right\rbrace  .
\end{equation}
We immediately see that $W_\varepsilon\to \{x_k\}$ in the Kuratowski sense as $\varepsilon\to 0$ since by \eqref{inclusion level sets}
\begin{equation}\label{inclusions W_e} 
    \left\lbrace u_{h_k} (\cdot,t_k)>l_k-\varepsilon  \right\rbrace \setminus\left\lbrace \eta ^\sigma_{h_k}(\cdot)> l_k +\varepsilon  \right\rbrace 
 \subseteq W_\varepsilon  \subseteq \left\lbrace u_{h_k} (\cdot,t_k)\ge l_k-\varepsilon  \right\rbrace \setminus\left\lbrace \eta ^\sigma_{h_k}(\cdot)> l_k +\varepsilon \right\rbrace,
\end{equation}    
 see also \eqref{eq 6.24 Nonlocal} below. Then, we check that $|W_\varepsilon|>0$ for all $\varepsilon$ small enough. By the continuity of $\eta^\sigma$ and $|\nabla \eta(\bar z)|\neq 0,$ for any $\varepsilon$ there exist a radius $r_\varepsilon$ such that $W_\varepsilon\supseteq B(x_k,r_\varepsilon)\cap T_{h,t_k-h_k}^+ E_{\varepsilon,k}$. Furthermore, for any $\varepsilon>0$, using  \eqref{inclusion level sets} again yields  $x_k\in T^+_{h_k,t_k-h_k}\{ u_{h_k}(\cdot,t_k-h_k)\ge l_k-\varepsilon \},$ and the latter set coincides with the closure of its points of density 1 by Lemma \ref{lemma density estimate}. Thus, $x_k$ satisfies lower density estimates and so we conclude that $|W_\varepsilon| > 0$. Now, assume $E_{\varepsilon,k}$ is bounded. By minimality we have 
\begin{align}
    &P_\phi(T_{h,t_k-h_k}^+E_{\varepsilon,k}) +\dfrac 1{h_k}\int_{T_{h,t_k-h_k}^+E_{\varepsilon,k}} \sd^\psi_{E_{\varepsilon,k} }(x)\ud x +\int_{W_\varepsilon} F_{h_k}(x,t_k-h_k)\ud x\nonumber\\
    &\le P_\phi\left(\left(T_{h,t_k-h_k}^+E_{\varepsilon,k}\right) \cap \{ \eta ^\sigma_{h_k}> l_k+\varepsilon \}\right) + \dfrac 1{h_k}\int_{\left(T_{h,t_k-h_k}^+E_{\varepsilon,k}\right) \cap \{ \eta ^\sigma_{h_k}> l_k \} } \sd^\psi_{E_{\varepsilon,k}}. \label{eq 6.21 Nonlocal}
\end{align}
Adding to both sides the term $P_\phi\left(\{ \eta ^\sigma_{h_k}> l_k+\varepsilon \}\cup T_{h,t_k-h_k}^+E_{\varepsilon,k} \right)$ and using the submodularity \eqref{submod}, we obtain 
\begin{align*}
    P_\phi(\{ \eta ^\sigma_{h_k}> l_k+\varepsilon \}\cup W_\varepsilon) &-  P_\phi(\{ \eta ^\sigma_{h_k}> l_k+\varepsilon \})	+\dfrac 1{h_k}\int_{ W_\varepsilon} \sd^\psi_{E_{\varepsilon,k} }(x)\ud x\\
    &+\int_{W_\varepsilon} F_{h_k}(x,t_k-h_k)\ud x\le 0.
\end{align*}
By \eqref{eq 6.19 Nonlocal},  $ \{u_{h_k}(\cdot,t_k-h_k)\ge  l_k-\varepsilon \}  \subseteq  \{ \eta (\cdot,t_k-h_k)\ge l_k-c_k-\varepsilon \} $, therefore it holds
\begin{align}
    P_\phi(\{ \eta ^\sigma_{h_k}> l_k+\varepsilon \}\cup W_\varepsilon) &- P_\phi(\{ \eta ^\sigma_{h_k}> l_k+\varepsilon \})	+\dfrac 1{h_k}\int_{ W_\varepsilon} \sd^\psi_{\{\eta (\cdot,t_k-h_k)\ge l_k-c_k-\varepsilon \} }(x)\ud x\nonumber\\
    &+\int_{W_\varepsilon} F_{h_k}(x,t_k-h_k)\ud x\le 0.\label{eq 6.22 Nonlocal}
\end{align}
If instead $E_{\varepsilon,k}$ is an unbounded set with compact boundary, we replace inequality \eqref{eq 6.21 Nonlocal}  by
\begin{align*}
    &P_\phi(T_{h,t_k-h_k} E_{\varepsilon,k}) +\dfrac 1{h_k}\int_{\left(T_{h,t_k-h_k}^+ E_{\varepsilon,k}\right)\cap B_R} \sd^\psi_{E_{\varepsilon,k}}(x)\ud x +\int_{W_\varepsilon} F_{h_k}(x,t_k-h_k)\ud x\\
    &\le P_\phi(\left( T_{h,t_k-h_k}^+ E_{\varepsilon,k}\right) \cap \{ \eta ^\sigma_{h_k}> l_k+\varepsilon \}) + \dfrac 1{h_k}\int_{ \left( T_{h,t_k-h_k}^+ E_{\varepsilon,k}\right) \cap \{ \eta ^\sigma_{h_k}> l_k+\varepsilon \}\cap B_R } \sd^\psi_{E_{\varepsilon,k} },
\end{align*}
for $R>0$ sufficiently large, see \eqref{eq 6.10 Nonlocal}. Then, one can argue as before to obtain \eqref{eq 6.22 Nonlocal}.\\
 {\textbf{Step 2.} We estimate the first two terms in  \eqref{eq 6.22 Nonlocal}.} The quantity $P_\phi(\{ \eta ^\sigma_{h_k}> l_k+\varepsilon \}\cup W_\varepsilon) - P_\phi(\{ \eta ^\sigma_{h_k}> l_k+\varepsilon \})$ can be estimated as done  in Lemma \ref{lemma estimates on balls}. Indeed, we consider the vector field $v=\nabla_p\phi(x,\nabla\eta ^\sigma_{h_k})$ in \eqref{P via calibration} and we use the divergence theorem to get 
\begin{equation}
    \begin{split}
       P_\phi(\{ \eta ^\sigma_{h_k}> l_k+\varepsilon \}\cup W_\varepsilon) - P_\phi(\{ \eta ^\sigma_{h_k}\ge l_k+\e \} ) &\ge \int_{\bd(\{ \eta ^\sigma_{h_k}> l_k+\varepsilon \}\cup W_\varepsilon)}v\cdot \nu-\int_{\bd\{ \eta ^\sigma_{h_k}> l_k+\e \}}  v\cdot\nu\\
    &=|W_\varepsilon|\fint_{W_\varepsilon} \div\, v 
    ,\label{eq 6.30 Nonlocal}   
    \end{split}
\end{equation}
where $\nu$ denotes the unit outer vector to the set we are integrating on. We then remark that $\fint_{W_\varepsilon} \div\, v\to H^\phi_{\{ \eta ^\sigma_{h_k}> l_k \}}(x_k)$ and $\fint_{W_\varepsilon} F_{h_k}(x,t_k-h_k)\ud x\to F_{h_k}(x_k,t_k-h_k)$ as $\varepsilon\to 0$ by continuity. \\ 
 {\textbf{Step 3.} We bound the distance term in \eqref{eq 6.22 Nonlocal} by showing that
\begin{equation}\label{eq 6.29 Nonlocal}
    \dfrac 1{h_k}\sd^\psi_{\{ \eta (\cdot,t_k-h_k)=l_k-c_k-\varepsilon \}}(z) \ge  \dfrac{\bd_t \eta (z,t_k)- O(h_k)}{ \psi(y,-\nabla\eta (y,t_k-h_k)) +O(h_k)}.
\end{equation}}
For any $z\in W_\varepsilon$, we have
\begin{equation}
    \eta (z,t_k)+c_k+\frac \sigma 2 |z-x_k|^2\le l_k +\varepsilon \label{eq 6.23 Nonlocal}.
\end{equation}
Since, in turn, $\eta (z,t_k)+c_k> l_k-\varepsilon$ it follows that  $ \sigma  |z-x_k|^2<4\varepsilon$ and thus, for $\varepsilon$ small enough,
\begin{equation}\label{eq 6.24 Nonlocal}
    W_\varepsilon\subseteq B_{c\sqrt \varepsilon}(x_k). 
\end{equation}
By a Taylor expansion, for every $z\in W_\varepsilon$ we have
\begin{equation}\label{eq 6.25 Nonlocal}
    \eta (z,t_k-h_k)=\eta (z,t_k)-h_k\bd_t \eta (z,t_k)+h^2_k\int_0^1 (1-s)\bd^2_{tt}\eta (z,t_k-s h_k)\ud s.
\end{equation}
Then, we consider $y,y_e\in\{ \eta (\cdot,t_k-h_k)(y)=l_k-c_k-\varepsilon \}$ being respectively, a point of minimal $\psi-$distance and Euclidean distance from $z$.\\
 {\textit{Claim:} We claim that it holds}  
\begin{equation}\label{decay y}
    |z-y|=O(h_k).
\end{equation}
In order to prove this result, we start remarking that for  $k\to\infty$ and choosing  $\varepsilon\ll h_k$, one has  $\sd^\psi_{\{ \eta (\cdot,t_k-h_k)\ge l_k-c_k-\varepsilon \}   }(z)\to 0$ (as $z\to x_k$ for $\varepsilon\to 0$ and $x_k\in \{\eta (\cdot,t_k)\ge l_k-c_k  \}$). In particular, recalling the bounds \eqref{inclusione psi palle} one has 
\[  |z-y_e|\le  c_\psi^2 |z-y|\le c_\psi^3 |\sd^\psi_{\{ \eta (\cdot,t_k-h_k)\ge l_k-c_k-\varepsilon \}   }(z)|\to 0  \]
as $k\to \infty.$  By \eqref{eq 6.23 Nonlocal} we deduce in particular $\eta (z,t_k)+c_k<l_k+ \varepsilon$, that is, 
\begin{equation}\label{eq 6.28 Nonlocal}
    0\le \eta (z,t_k)-\eta (y,t_k-h_k)\le 2\e,
\end{equation}
and the same inequality substituting $y_e$ to $y$. Thus, one has 
\begin{align*}
    \eta (z,t_k)-\eta (y_e,t_k-h_k)&=\nabla\eta (y,t_k-h_k)\cdot (z-y_e)-h_k\bd_t\eta (y,t_k-h_k)+O(|z-y_e|^2+h_k^2)
\end{align*}
which we combine with $\nabla\eta (y,t_k-h_k)\cdot (z-y_e)=\pm|\nabla\eta (y,t_k-h_k)|\, |z-y_e| $ (see \cite{ChaMorPon15} for details) and \eqref{eq 6.28 Nonlocal} to get
\[ |z-y_e|\,|\nabla\eta (y,t_k-h_k)|\le 2\e +O(h_k)+O(|z-y_e|^2).\]
Recalling that $|\nabla\eta (y,t_k-h_k)|\ge c>0$ for $h_k$ small enough, we divide by $|\nabla\eta (y,t_k-h_k)|$ to conclude $|z-y_e|=O(h_k)$ as $\e\ll h_k.$ Finally, employing again \eqref{inclusione psi palle}, we  {prove the claimed} \eqref{decay y}.

Then, we consider a  geodesic curve for the definition of $\sd^\psi_{\{ \eta (\cdot,t_k-h_k)\ge l_k-c_k-\varepsilon \}   }(z)$: if this distance is positive, we choose  $\gamma:[0,1]\to \R^N$ with $\gamma(0)=z, \gamma(1)=y$, with $y$ as before, otherwise we take $\gamma$ such that $\gamma(0)=y, \gamma(1)=z$. In the following, we will assume $\sd^\psi_{\{ \eta (\cdot,t_k-h_k)\ge l_k-c_k-\varepsilon \}   }(z)>0$, the other case being analogous. Recalling \eqref{CS type ineq}, we have
\begin{align*}
    \eta (z,t_k-h_k)&=\eta (y,t_k-h_k)+\int_0^1 \nabla \eta (\gamma,t_k-h_k)\cdot \dot\gamma\ud t\\
    &\ge \eta (y,t_k-h_k)-\int_0^1 \psi(\gamma,-\nabla\eta  (\gamma,t_k-h_k ))\psi^\circ(\gamma,\dot\gamma)\ud t\\
    &\ge \eta (y,t_k-h_k)-\psi(y,-\nabla\eta  (y,t_k-h_k))\,\sd^\psi_{\{ \eta (\cdot,t_k-h_k)=l_k-c_k-\varepsilon \}   }(z)\\
    &\quad -\int_0^1 \left( \psi(\gamma,-\nabla\eta  (\gamma,t_k-h_k))-\psi(y,-\nabla\eta  (y,t_k-h_k))  \right)\psi^\circ(\gamma,\dot\gamma)\ud t \\
    &\ge \eta (y,t_k-h_k)-\big(\psi(y,-\nabla\eta  (y,t_k-h_k))  + c|z-y| \big)\,\sd^\psi_{\{ \eta (\cdot,t_k-h_k)=l_k-c_k-\varepsilon \}   }(z),
\end{align*}
where in the last line we reasoned as in \eqref{estimate curve} to obtain the bound $\sup_t|\gamma(t)-y|\le c|z-y|$. 
Recalling \eqref{decay y} one has 
\begin{equation}\label{eq 6.26 Nonlocal} 
    \eta (z,t_k-h_k)\ge \eta (y,t_k-h_k)-\psi(y,-\nabla\eta (y,t_k-h_k))\,\sd^\psi_{\{ \eta (\cdot,t_k-h_k)=l_k-c_k-\varepsilon \}   }(z)+o(h_k).
\end{equation}
Combining  \eqref{eq 6.25 Nonlocal} with \eqref{eq 6.26 Nonlocal} and using \eqref{eq 6.28 Nonlocal}, we deduce 
\begin{align*}
    \sd^\psi&_{\{ \eta (\cdot,t_k-h_k)=l_k-c_k-\varepsilon \}}(z)\,  \psi(y,-\nabla\eta (y,t_k-h_k))+o(h_k)\\
    & \ge -2\varepsilon +h_k\bd_t \eta (z,t_k)-h_k^2 \int_0^1(1-s)\bd^2_{tt}\eta (z,t_k-sh_k)\ud s.
\end{align*}
Note that, in view of \eqref{eq 6.23 Nonlocal} and \eqref{inclusione psi palle}, $|\eta (z,t_k)-\eta (y,t_k)|\le c\varepsilon+ch_k=O(h_k)$, provided $\varepsilon \ll h_k $ and small enough.  {We then conlude \eqref{eq 6.29 Nonlocal} by combining the previous inequality with} \eqref{eq 6.24 Nonlocal},\eqref{decay y} as
\begin{align*}
    \dfrac 1{h_k}  \sd^\psi_{\{ \eta (\cdot,t_k-h_k)=l_k-c_k-\varepsilon \}}(z) &\ge \dfrac {\bd_t \eta (z,t_k)-\frac{2\varepsilon}{h_k}-O(h_k)-O_{h_k}(1)   }{\psi(y,-\nabla\eta (y,t_k-h_k))}   \nonumber\\
    & = \dfrac {\bd_t \eta (x_k,t_k)+O(\sqrt\varepsilon)-\frac{2\varepsilon}{h_k}-O(h_k)-O_{h_k}(1)   }{\psi(x_k,-\nabla\eta (x_k,t_k-h_k))+O(\sqrt\varepsilon)+O(h_k)}.
\end{align*}
\noindent  {\textbf{Step 4.}} We conclude  {the proof} by employing \eqref{eq 6.22 Nonlocal}, \eqref{eq 6.30 Nonlocal}  and \eqref{eq 6.29 Nonlocal}, dividing by $|W_\varepsilon|$ and sending $\varepsilon\to 0$ to obtain
\[ \dfrac{\bd_t \eta (x_k,t_k)-O_{h_k}(1)}{\psi(x_k,-\nabla\eta (x_k,t_k))+O(h_k)}+H^\phi_{\{ \eta ^\sigma_{h_k}\ge \eta ^\sigma_{h_k}(x_k) \}}  (x_k)-F_{h_k}(x_k,t_k-h_k)\le 0. \]
Letting simultaneously $\sigma\to 0$ and $k\to \infty$, recalling the continuity properties of $H^\phi$, we deduce \eqref{eq viscosa}. Indeed the sets $\{ \eta ^\sigma_{h_k}> \eta ^\sigma_{h_k}(x_k) \}$ are converging in $C^2$ to the set $\{ \eta > \eta  (x) \}$, $x_k\to x$ and thus 
\[ H^\phi_{\{ \eta ^\sigma_{h_k}> \eta ^\sigma_{h_k}(x_k) \}}  (x_k)\to H^\phi_{\{ \eta > \eta (x) \}}  (x),\]
and we conclude the proof of this step.\\
 {\textbf{Case 2.}}  Now we consider the case $\nabla\eta (\bar x, \bar t)=0$ and we show that $\bd_t \eta (\bar x, \bar t)\le 0$. The proof follows the line of the one in \cite{ChaMorPon15}, we just highlight the differences. 

Since $\nabla \eta (\bar z)=0,$ there exist  { $a,b>0$} such that 
\[|\eta (x,t)-\eta (\bar z)-\bd_t\eta (\bar z)(t-\bar t)|\le   {a|x-\bar x|^3+b|t-\bar t|^2,} \]
thus, we can define
\begin{align*}
    &\tilde\eta (x,t)=\bd_t\eta (\bar z)(t-\bar t) +2 {a|x-\bar x|^3}+2 {b|t-\bar t|^2}\\
    &\tilde\eta _k(x,t)=\tilde\eta (x,t)+\dfrac 1{k(\bar t - t)}.
\end{align*}
We remark that $u-\tilde\eta $ achieves a strict maximum in $\bar z$ and  the local maxima of $u-\tilde\eta _k$ in $\R^N \times [0,\bar t]$ are in points $(x_k,t_k)\to\bar z$ as $k\to \infty$, with $t_n\le \bar t$. From now on, the only difference from \cite{ChaMorPon15} is in the case $x_k= \bar x$ for an (unrelabeled) subsequence. We assume $x_k=\bar x$ $\forall k>0$ and define $b_k=\bar t-t_k>0$ and the radii
 {\begin{equation*}
    r_k:={2\sqrt{Cb_k}},
\end{equation*}
where $C$ is the constant of Lemma \ref{lemma estimate decay balls}. Taking $k$ large enough, by Lemma \ref{lemma estimate decay balls} the balls $B(\cdot,r_k)$ have an extinction time greater than $ 2(\bar t-t_k)$.  We then have
\begin{align*}
    B(\bar x,r_k)&\subseteq  \{ \tilde\eta _k(\cdot,t_k)\le   \tilde\eta _k(\bar x,t_k)+2 {a r_k^3} \}\\&\subseteq  \{ u(\cdot,t_k)\le   u(\bar x,t_k)+2{a r_k^3}  \},
\end{align*}
by maximality of $u-\tilde\eta _k$  at $z_k$. Since the balls $B(\cdot,r_k)$ are not vanishing, we conclude
\[ \bar x\in \{ u(\cdot,\bar t)\le   u(\bar x,t_k)+2{a r_k^3}  \}. \]
Finally, we use again the maximality of $u-\eta $ at $\bar z$ and the choice of $r_k$ to obtain
\[  \dfrac{\eta (\bar x,t_k)-\eta (\bar z)}{t_k-\bar t} = \dfrac{\eta (\bar x,t_k)-\eta (\bar z)}{-b_k}\le  \dfrac{u(\bar x,t_k)-u(\bar x,\bar t)}{-b_k}\le \dfrac{-2{a r_k^3}}{-b_k}=c\, \sqrt{b_k}. \]}
Passing to the limit $k\to\infty$, we conclude that $\bd_t \eta (\bar z)\le 0$.
\end{proof}

We conclude with two remarks concerning some possible generalizations of the results presented.
\begin{oss}
   The  results presented in this work can be immediately extended to unbounded initial open sets $E_0$, whose boundary is compact. Indeed, defining the discrete flow as $E_t\h=E_0$ if $t\in [0,h)$, otherwise by induction $E_{t}\h=T^-_{h,t}E_{t-h}\h,$ where the operator $T_{h,n}^-$ is the one defined in \eqref{pb discr compl}, this evolution is uniquely characterized by the one of the complement. Thus, all the results presented in this paper can be extended to this particular unbounded case.
\end{oss}

\begin{oss}\label{rmk rel geom}
    Following the lines of \cite{BelPao} (in the spirit of \cite{AlmTayWan}) one can see that the results of this paper may be extended to prove existence of flat flows and level set solutions to the mean curvature flow    on $\R^N$ endowed with the geometric structure induced by a Finsler metric $\phi^\circ.$  For example, the perimeter functional in this setting is defined as follows. Given a set $E$ of finite perimeter, its (intrinsic) perimeter is 
    \[\mathcal P_{\phi^\circ} (E)=\int_{\bd^* E} \phi(x,\nu_E(x)) \udH_{\phi^\circ}(x),\]
    where the Hausdorff measure $\mathcal H^{N-1}_{\phi^\circ}$ is the one induced by the metric ${\phi^\circ}.$ In particular, one can compute  $\udH_{\phi^\circ}(x)=\omega_N |B^{\phi^\circ}(x)|^{-1}\udH(x)$  (see \cite{BelPao}), thus this approach is equivalent to consider in our framework a slightly different (but still regular) anisotropy, namely $\phi^*(x,\nu):=\omega_N|B^{\phi^\circ}(x)|^{-1}\phi(x,\nu)$.  In particular, this approach leads to considering the evolution of hypersurfaces $E_t$ moving according to the evolution law 
    \[  V_{\phi^\circ}(x,t)= -\mathcal H_{E_t}(x)+f(x,t)\quad x\in\bd E_t,\ t\in (0,T) \]
    where now $V_{\phi^\circ}$ represents the speed of evolution along the anisotropic normal outer vector $n_{\phi^\circ}(x)=\nabla_p \phi(x,\nu_E(x))$ and  $ \mathcal H$ is the \virg{intrinsic} mean curvature, thus the first variation of the perimeter $\mathcal P_{\phi^\circ}$. Recalling that $n_{\phi^\circ} (x) \cdot \nu_E(x)=\phi(x,\nu_E(x))$,  we see that the hypersurfaces are evolving with a normal (in the Euclidean sense) velocity given by the law $$V(x,t)=\phi(x,\nu_{E_t}(x)) \left( -H^{\phi^*}_{E_t}(x)+f(x,t)\right).$$
    After this transformation, we can apply the results previously proved.
\end{oss}
\subsection*{Acknowledgements}
 {The authors want to thank the referees for the careful reading of the manuscript and their comments, which helped improve the paper.} D.~De Gennaro has received funding from the European Union's Horizon 2020 research and innovation programme under the Marie Skłodowska-Curie grant agreement No 94532 \includegraphics[scale=0.01]{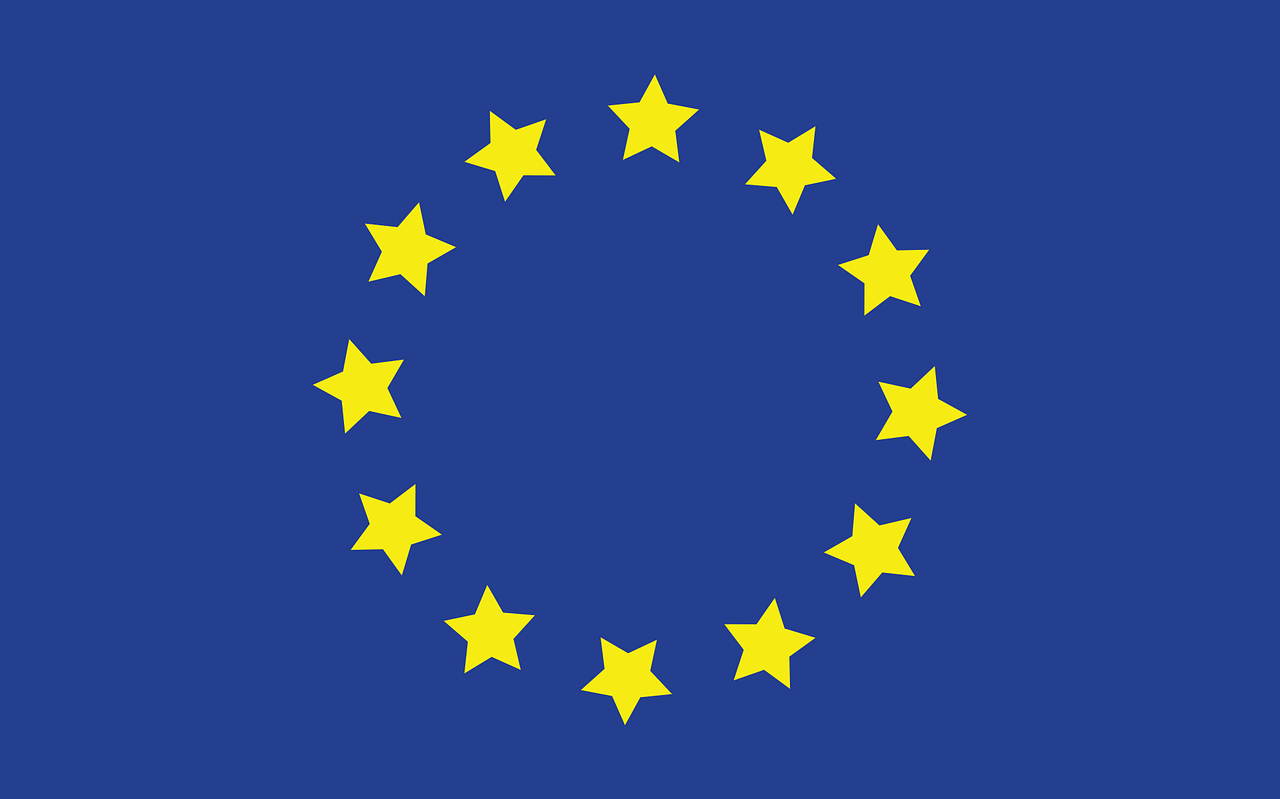}.

\printbibliography 
\end{document}